\documentclass{article}%
\usepackage{amssymb}
\usepackage{amsmath,color}%
\usepackage{amsfonts}%
\usepackage{graphicx}
\usepackage{authblk} 
\providecommand{\U}[1]{\protect\rule{.1in}{.1in}}
\newtheorem{theorem}{Theorem}

\newtheorem{corollary}[theorem]{Corollary}

\newtheorem{definition}[theorem]{Definition}
\newtheorem{example}[theorem]{Example}

\newtheorem{lemma}[theorem]{Lemma}

\newtheorem{proposition}[theorem]{Proposition}
\newtheorem{remark}[theorem]{Remark}

\newenvironment{proof}[1][Proof]{\noindent\textbf{#1} }{\ \rule{0.5em}{0.5em}}

\newcommand{\beq} {\begin{eqnarray*}}
\newcommand{\eeq} {\end{eqnarray*}}

\def \R{\mathbb{R}}

\def \N{\mathbb{N}}

\def \F{\mathbb{F}}
\def \G{\mathbb{G}}
\def \E{\mathbb{E}}
\def \M{\mathbb{M}}
\def \P{\mathbb{P}}

\newcommand{\1}{{1\hspace{-0.2ex}\rule{0.12ex}{1.61ex}\hspace{0.5ex}}}

\title{A Central Limit Theorem for Wasserstein type distances between two different laws} 
\date{} 
\author[1]{Philippe Berthet}\author[2]{Jean Claude Fort} \author[1,3]{Thierry Klein\thanks{Corresponding author \texttt{thierry.klein@math.univ-toulouse.fr } } } 
\affil[1]{Institut de math\'ematique, UMR5219; Universit\'e de Toulouse;  CNRS, UPS IMT, F-31062 Toulouse Cedex 9, France} 
\affil[2]{MAP5, Universit\'e Paris Descartes, France} 
  \affil[3]{ENAC - Ecole Nationale de l'Aviation Civile, Universit\'e de Toulouse, France}

\begin{document}

\maketitle 

\begin{abstract}
This article is dedicated to the estimation of Wasserstein distances and Wasserstein costs between two distinct continuous distributions  $F$ and $G$ on $\mathbb R$. The estimator is based on the order statistics of  (possibly dependent) samples of $F$  resp. $G$. We prove the consistency  and the asymptotic normality of our estimators. \\
\begin{it}Keywords:\end{it} \\Central Limit Theorems-
Generelized Wasserstein distances-
Empirical processes-
Strong approximation-
Dependent samples.\\
\begin{it} MSC Classification:\end{it}\\62G30, 62G20, 60F05, 
60F17 
\end{abstract}

\tableofcontents
\section{Introduction}\label{sec:intro}
\subsection{Motivation}
In this article we address the problem of estimating the distance between two different distributions with respect to a class of Wasserstein costs that we define in the sequel. The framework is very simple:  two samples of $i.i.d.$ real random variables taking values in $\mathbb R$  with continuous cumulative distribution function ($c.d.f.$) $F$ and $G$ are available. These samples are not necessarily independent, for instance they may be issued from simultaneous experiments. From these samples we estimate the Wasserstein distances or costs between $F$ and $G$ and we prove a central limit theorem (CLT).

The motivation of this work is to be found in the fast development of computer experiments. Nowadays the output of many computer codes is not only a real multidimensional variable but frequently a function computed on so many points that it can be considered as a functional output. In particular this function may be the density or $c.d.f.$ of a real random variable. To analyze such outputs one  needs to choose a distance to compare various $c.d.f.$. Among the large possibilities offered by the literature   the Wasserstein distances are now commonly used - for more details on general Wasserstein distances we refer to \cite{villani2003tot}. Since computer codes only provide samples of the underlying distributions, the estimation of such distances are of primordial importance.  Actually for univariate probability distributions the $p$-Wasserstein distance  simply is the $L^p$ distance of simulated random variables from a common and universal (uniform on $[0,1]$) simulator  $U$: $W_p^p(F,G)=\int_0^1|F^{-1}(u)-G^{-1}(u)|^p du=\mathbb E |F^{-1}(U)-G^{-1}(U)|^p$, where $F^{-1}$ is the generalized inverse of $F$. It is then natural to estimate $W_p^p(F,G)$ by its empirical counterpart that is $W_p^p(\mathbb{F}_n,\mathbb{G}_n)$ where $\mathbb{F}_n$ and $\mathbb{G}_n$ are the empirical $c.d.f.$ of $F$ and $G$ build through $i.i.d.$ samples of $F$ and $G$ - the two samples are possibly dependent.

Many authors were interested in the convergence of $W_p^p(\mathbb{F}_n,F)$, see $e.g.$ the survey paper \cite{Ledoux17} or  \cite{Delbarrio05,Delbarrio99,Delbarrio03,Delbarrio11}. Up to our knowledge there are only two  recent works studying the convergence of $W_2^2(\mathbb{F}_n,\mathbb{G}_n)$ \cite{dBL,Munk}. In \cite{dBL} very general results are obtained in the multivariate setting when the two samples are independent. However the estimator in not explicit from the data, the centering in the CLT is $\mathbb EW_2^2(\mathbb F_n,\mathbb G_n)$ rather than $W_2^2(F,G)$ itself, and the limiting variance is also not explicit. In \cite{Munk} multivariate discrete distributions and $W_2$ cost are considered, again for independent samples, and the CLT is explicit.\\
To investigate more deeply the univariate setting we allow a larger class of convex costs and also dependent $i.i.d.$ samples from continuous $c.d.f.$ $F$ and $G$. We look for an explicit CLT for the easily computed natural estimator, under almost minimal conditions relating $F$ and $G$ to the cost.

\subsection{Setting}
Let $F$ and $G$ be two $c.d.f.$ on $\mathbb R$. The $p$-Wasserstein distance between $F$ and $G$ is defined to be
\begin{equation}\label{Wass}
W^p_p(F,G)= \min_{X\sim F,Y\sim G}\mathbb E|X-Y|^p,
\end{equation}
where $X\sim F,Y\sim G$ means that X and Y are  random variables with respective c.d.f. $F$ and $G$.
The minimum in (\ref{Wass}) has the following  explicit expression

\begin{equation}\label{Frechet}
W^p_p(F,G)= \int_0^1 |F^{-1}(u)-G^{-1}(u)|^p du. 
\end{equation}

\noindent The Wassertein distances can be generalized to Wasserstein costs. Given a real non negative function $c(x,y)$ of two real variables, we consider the Wasserstein cost
\begin{equation}\label{Wassc}
W_c(F,G)= \min_{X\sim F,Y\sim G}\mathbb E c(X,Y).
\end{equation}
We restrict our study to costs for which  this minimum is finite and  the analogue of \eqref{Frechet} exists.
\begin{definition}
We call a good cost function any application $c$  from $\mathbb R^{2}$ to $\mathbb R$ that defines a negative measure on $\mathbb{R}^2$. It satisfies the "measure  property" $\mathbf{\mathcal P}$,
$$
\mathbf{\mathcal P}: \forall x\leqslant x'\ \mathrm{and\ } \forall y\leqslant y',\  c(x',y')-c(x',y)-c(x,y')+c(x,y)\leqslant 0.
$$
\end{definition}
\begin{remark}\label{cpositive} It is obvious that  $c(x,y)=-xy$ satisfies the $\mathbf{\mathcal P}$ property and  if $c$ satisfies $\mathbf{\mathcal P}$ then any function of the form $a(x)+b(y)+c(x,y)$ also  satisfies $\mathbf{\mathcal P}$. In particular $(x-y)^2=x^2+y^2-2xy$ satisfies $\mathbf{\mathcal P}$.
 More generally if $\rho$ is a convex real function then $c(x,y)=\rho(x-y)$ satisfies $\mathbf{\mathcal P}$. This is the case of  $|x-y|^p$, $p\geqslant 1$ and for the cost associated to the $\alpha$-quantile $c(x,y)=(x-y) (\alpha-{\bf 1}_{x-y<0})$. 
\end{remark}
The following theorem that can be found in \cite{Cambanis76}  gives an explicit formula of $W_c$ for    cost functions satisfying property $\mathcal P$.

\begin{theorem}[Cambanis, Simon, Stout \cite{Cambanis76}]\label{theo:cam} Let $c$  satisfy the "measure  property"  $\mathcal P $ and  $U$ be a random variable uniformly distributed on $[0,1]$, then 
$$ W_c(F,G)=\int_0^1 c(F^{-1}(u),G^{-1}(u))du=\mathbb E\ c(F^{-1}(U),G^{-1}(U)).$$
\end{theorem}

In view of Theorem \ref{theo:cam}, an estimator of  $W_{c}(F,G)$ based on a sample from  the joint distribution of $(F^{-1}(U),G^{-1}(U))$ seems the most natural one. Nevertheless, it is not necessary and one can sample from any coupling of the marginal c.d.f. This is very interesting  in  practice, since we can use  experimental data without any assumption on the coupling structure. We will see that it only affects the limiting variance in the CLT but not the rate of convergence. \\

\noindent Let $\displaystyle (X_i,Y_i)_{1\leqslant i\leqslant n}$ be an   $i.i.d.$ sample of a random vector with distribution $\Pi$ and marginal c.d.f. $F$ and $G$. 
Write $\F_n$ and $\G_n$ the random empirical $c.d.f.$ built from the two marginal samples. Let $c$ a good cost function. Denote by $X_{(i)}$ (resp.  $Y_{(i)}$) the $i^{th}$ order statistic of the sample $\displaystyle (X_i)_{1\leqslant i\leqslant n}$ (resp. $\displaystyle (Y_i)_{1\leqslant i\leqslant n}$), i.e.  $X_{(1)}\leqslant \ldots\leqslant X_{(n)}$. We have
\begin{equation}
W_c(\F_n,\G_n)=\frac{1}{n}\sum_{i=1}^n c(X_{(i)},Y_{(i)}). 
\end{equation}
Thanks to Theorem \ref{theo:cam},  $W_c(\F_n,\G_n)$ is a natural estimator of 
$W_c(F,G)$. The aim of this paper is to study its asymptotic properties when $F \ne G$ and $F$ and $G$ are continuous. 
Our main result is the weak convergence of  
\begin{equation}
\sqrt n \left(W_{c}(\F_{n},\G_{n})-W_{c}(F,G)\right).\label{Lambdan}
\end{equation}

\subsection{Overview of the paper.}

\noindent\textbf{Organization.} The paper is organized as follows. Assumptions are discussed in Section \ref{sec:hyp}. In Section  \ref{sec:main} we state our main result in the form of a CLT for 
$\displaystyle\sqrt n \left(W_{c}(\F_{n},\G_{n})-W_{c}(F,G)\right)$. A few prospects are presented in Section  \ref{sec:concl}.
All the results are proved in Section  \ref{sec:proof}. Section  \ref{sec:app} contains the proofs of technical results used in the previous section and complements on the assumptions. \smallskip

\noindent\textbf{About the assumptions.} In order to control the integrals $W_{c}(F,G)$ and $W_{c}(\F_{n},\G_{n})$
we separate out three sets of assumptions. First,
about the regularity of $F$ and $G$ and  the separation of their tails, with the convention that $G$ has a
lighter tail. Second, on the
rate of increase, the regularity, the asymptotic expansion of $c$ and the behaviour of $c(x,y)$ close to the diagonal $y=x$. The first
two sets are hereafter labelled $(FG)$ and $(C)$ respectively. They allow to
separately select a class of probability laws and an admissible cost. The
third set is labelled $(CFG)$ and mixes the requirements on $(F,G,c)$ making
them compatible. \\
Conditions $(C)$ encompass a large class of good Wasserstein 
costs $c$, but $W_1$  is not included - see remark \ref{W1} below. Conditions $(FG)$ are satisfied by all classical laws of
probability. It is important to point out  that conditions $(FG)$ and $(CFG)$ are free from the joint law $\Pi$ of the two samples.
Given a cost $c$ satisfying conditions $(C)$, conditions $(FG)$ and $(CFG)$
provide sufficient regularity and tail conditions on $F$ then regularity, tail
and closeness conditions on $G$. The nice feature is that $(CFG)$ are almost
minimal to ensure that the limiting variance $\sigma^2$ satisfies $\sigma^{2}(\Pi,c)<+\infty$ whatever the joint law,
hence for our CLT.\smallskip

\noindent\textbf{Method.} The $(F,G,c)$-dependent technique of proof we
propose consists in two major steps. At the first step we combine the
assumptions to show that extreme tail terms and approximations in (\ref{Lambdan}) can
be neglected in probability. Next, large quantiles can be centered on a larger scale and their
deviation is led by the two marginal empirical quantile processes. All the
assumptions $(C)$, $(FG)$ and $(CFG)$ are required to control the outer
integral error processes at the $\sqrt{n}$ rate. At the second step, since
only the most central part of integrals eventually matters in (\ref{Lambdan})
it remains to prove its weak convergence to a Gaussian law. At this stage the
pertinent tool is a Brownian approximation of joint non extreme quantiles. The
joint distribution naturally shows up together with the CLT rate $\sqrt{n}$.

\begin{remark}\label{W1}

\noindent The distance $W_1$ does not satisfy assumption $(C3)$ since the derivative of the absolute value does not vanish at $0$. This is a meaningful border case since the limiting law may now depend on the set $\{F=G\}$. 
\end{remark}

\section{Notation and assumption}\label{sec:hyp}

\subsection{Notation}
Let $H$ denote the bivariate distribution function of $\Pi$, thus%
\[
H(x,y)=\mathbb{P}(X\leqslant x,Y\leqslant y),\quad F(x)=H(x,+\infty),\quad
G(y)=H(+\infty,y).
\]
For the sake of clarity, we focus on the generic case where the c.d.f 
$F$ and $G$ have positive densities  $f=F^{\prime}$ and $g=G^{\prime}$ supported on the whole line $\R$.  Write $F^{-1}$ and
$G^{-1}$ their quantile functions. The tail
exponential order of decay are defined to be%
\begin{equation}
\psi_{X}(x)=-\log\mathbb{P}(X>x),\quad\psi_{Y}(x)=-\log\mathbb{P}(Y>x),\quad
x\in\mathbb{R}_{+},\label{psy}%
\end{equation}
We introduce  the density quantile functions%
\[
h_{X}=f\circ F^{-1},\quad h_{Y}=g\circ G^{-1},
\]
and  their companion functions%
\[
H_{X}(u)=\frac{1-u}{F^{-1}(u)h_{X}(u)},\quad H_{Y}(u)=\frac{1-u}%
{G^{-1}(u)h_{Y}(u)}.
\]
\noindent For $k\in\mathbb{N}_{\ast}$ denote $\mathcal{C}_{k}(I)$ the set of
functions that are $k$ times continuously differentiable on $I\subset\mathbb{R}$, and
$\mathcal{C}_{0}(I)$ the set of continuous functions. Let $\mathcal{M}_{2}\left(  m,+\infty\right)  $ be the subset of functions
$\varphi\in\mathcal{C}_{2}$ such that $\varphi^{\prime\prime}$ is monotone on $\left(  m,+\infty\right)$. 
Write
$RV(\gamma)$ the set of regularly varying functions at $+\infty$ with index $\gamma\geqslant 0$. 
We consider  slowly varying functions $L$ satisfying

\begin{equation}
L^{\prime}(x)=\frac{\varepsilon_{1}(x)L(x)}{x},\quad\varepsilon_{1}%
(x)\rightarrow0\text{ as }x\to +\infty.\label{L'}%
\end{equation}
This slight restriction is explained in the Appendix at Section \ref{regvar}. 
Then for integrability reasons we impose

\begin{equation}
L^{\prime}(x)\geqslant\frac{l_{1}}{x},\quad l_{1}\geqslant1.\label{L1}%
\end{equation}
When $\gamma=0$ we define%
\[
RV_{2}^{+}(0,m)=\left\{  L:L\in\mathcal{M}_{2}\left(  m,+\infty\right)
\text{\ such that (\ref{L'}) and (\ref{L1}) hold}\right\}  .
\]
When $\gamma>0$%
\[
RV_{2}^{+}(\gamma,m)=\left\{  \varphi:\varphi\in\mathcal{M}_{2}\left(  m,+\infty\right), \varphi(x)=x^\gamma L(x)\text{\ such that }L'\text{ obeys (\ref{L'})}\right\}  .
\]
\subsection{Assumption}
\subsubsection{Conditions $(FG)$.}

Let  $m>\max(0,F^{-1}(1/2),G^{-1}(1/2))$ be large enough to satisfies all the subsequent assumptions. Let $\overline{u}%
=\max(F(m),G(m))>1/2$. We assume that there   exists $\tau_0>0$ such that 
\begin{align*}
(FG1)&\quad F,G\in\mathcal{C}_{2}(\mathbb{R}_{+}\mathbb{)}\text{,}\quad
f,g>0\text{ on }\mathbb{R}_{+}.
\\
(FG2)&\quad
\left(  1-u\right)  \left\vert \left(
\log h(u)\right)  ^{\prime}\right\vert\textrm{\ is\ bounded on \ }(\bar u,1),\quad h=h_{X},h_{Y}.
\\
(FG3)&\quad H_{X},H_{Y}\text{ are bounded on } (\bar u,1).\\
(FG4)&\quad\ \tau(u)=F^{-1}(u)-G^{-1}(u)\geqslant\tau_{0},\quad u\geqslant \overline{u}.
\end{align*}
\begin{remark}
Assumption $(FG4)$ means that the right tails of $F$ and $G$ are asymptotically well separated. In particular it allows translation models. \\
Rewritting $(FG2)$ and $(FG3)$ with the density functions we get the following equivalent conditions
\begin{equation*}
(FG5) \sup_{x>m}\frac{1-F(x)}{f(x)}\left(\frac{1}{x}+\frac{|f'(x)|}{f(x)}\right)<\infty\mathrm{\ and\ } \sup_{x>m}\frac{1-G(x)}{g(x)}\left(\frac{1}{x}+\frac{|g'(x)|}{g(x)}\right)<\infty.
\end{equation*}
At Section \ref{SCFG} we provide  a sufficient condition for ($FG1$), ($FG2$), ($FG3$). 
\end{remark}

\begin{example}All classical probability laws with lighter tail
than a Pareto law are $(FG)$ admissible since they are smooth enough. An
example of heavy tail is the Pareto law with parameter $p>0$ for which
\[
\psi_{X}(x)=p\log x,\ F^{-1}(u)=(1-u)^{-1/p},\ H_{X}(u)=\frac{1}{p},
\]%
\[
h_{X}(u)=p(1-u)^{1+1/p},\ \left(  1-u\right)  \left\vert \left(  \log
h_{X}(u)\right)  ^{\prime}\right\vert =\frac{1}{p}.
\]
An example of light tail is the Weibull law with parameter $q>0$ for which
\[
\psi_{X}(x)=x^{q},\ F^{-1}(u)=\left(  \log(1/(1-u))\right)  ^{1/q}%
,\ H_{X}(u)=\frac{1}{q\log(1/(1-u))},
\]%
\[
h_{X}(u)=q(1-u)\left(  \log(1/(1-u))\right)  ^{1-1/q},\ \left(  1-u\right)
\left\vert \left(  \log h_{X}(u)\right)  ^{\prime}\right\vert \sim\frac{1}{q}%
\]
and this law is log-convex if $q<1$, log-concave if $q>1$. If $\psi_{X}$ is
regularly varying with index $q>0$ the previous functions are only modified by
a slowly varying factor, as for the Gaussian law. 
\end{example}

\subsubsection{Conditions $(C)$}
We consider smooth Wasserstein costs satisfying property $\mathbf{\mathcal P}$. 
We impose ($wlog$) that $c(x,x)=0$ and  assume that, for $0<\tau_1<\tau_0$ and some $\gamma\geqslant0$
\begin{align*}
({C1}) &\quad c(x,y)\geqslant0,\quad c\in\mathcal{C}_{1}( [-m,m]\times\mathbb R\cup \mathbb R\times[-m,m] ).\\
({C2})&\ c(x,y):=\rho\left(|x-y|\right)=\exp(l(\left\vert x-y\right\vert )), (x,y)\in\left(  m,+\infty\right)  ^{2}, l\in{RV}_{2}^{+}(\gamma
,\tau_1\mathbb{)}. \end{align*}
Thus $c$ is asymptotically smooth and symetric. Moreover we need the following contraction of $c\left(  x,y\right)$ along the
diagonal $x=y$. We assume that there exists  $d(m,\tau)\to 0$ as $\tau\to 0$ such that
\begin{align*}
({C3})\quad \left\vert c\left(  x^{\prime},y^{\prime}\right)  -c\left(  x,y\right)
\right\vert \leqslant d(m,\tau)\left(  \left\vert x^{\prime}-x\right\vert
+\left\vert y^{\prime}-y\right\vert \right) \mathrm{for } \left(  x,y\right), \left(
x^{\prime},y^{\prime}\right) \in 
D_{m}(\tau), 
\end{align*}
where
$D_{m}(\tau)=\left\{  (x,y):\max(\left\vert x\right\vert ,\left\vert
y\right\vert )\leqslant m,\left\vert x-y\right\vert \leqslant\tau\right\}  
$. 
\begin{remark} Under $(C2)$ we have%
\[
\rho(\left\vert x-y\right\vert )\leqslant\rho(\max(x,y))\leqslant\rho
(x)+\rho(y),\quad(x,y)\in\left(  m,+\infty\right)  ^{2}.
\]
 Hence%
\begin{equation}
\sup_{x>m,y>m}\frac{c(x,y)}{\rho(x)+\rho(y)}\leqslant1.\label{crhosup}%
\end{equation}
\end{remark}

\begin{example}Typical costs satisfying the conditions $(C)$ are,
for $\alpha>1$,
\begin{equation}
c_{\alpha}(x,y)=\left\vert x-y\right\vert ^{\alpha}\label{calpha}
\end{equation}
and, for $\beta>0$,%
\begin{equation}
c_{\beta}^{-}(x,y)=\exp((\log\left(  1+\left\vert x-y\right\vert \right)
)^{1+\beta})-1,\quad c_{\beta}^{+}(x,y)=\exp(\left\vert x-y\right\vert ^{\beta
})-1.\label{cbeta}%
\end{equation}
They satisfy $(C2)$ with $\gamma=0$, $\gamma=0$ and $\gamma=\beta$
respectively.
\end{example}

\subsubsection{Conditions $(CFG)$}
Recall that if $(C2)$ holds the)  $ l\in{RV}_{2}^{+}(\gamma
,\tau_1\mathbb{)}$. Now when $\gamma=0$
in order to compare the tail functions and the cost function we need %
\begin{equation}
\underset{x\rightarrow+\infty}{\lim\sup}\frac{\log(xl^{\prime}(x))}{\log
l(x)}=1-\underset{x\rightarrow+\infty}{\lim\inf}\frac{\log(1/\varepsilon
_{1}(x))}{\log l(x)}=\theta_{1}\in\left[  0,1\right] ,\label{theta1}%
\end{equation}
where $\varepsilon_{1}$ is defined in \eqref{L'}.  In the case $\gamma>0$ we set  $\theta_1=1$. The following crucial assumption $(CFG)$ connects the distribution's tails with the  cost function.

\begin{align*}
(CFG)&\textrm{\ There exists\  } \theta>1+\theta_1\textrm{\ such\ that\ }\left(  \psi_{X}\circ l^{-1}\right)  ^{\prime}(x)\geqslant
2+\frac{2\theta}{x},\ x\geqslant
l(\tau_1).
\end{align*}
\begin{remark}
For Wasserstein distances given by $c_{\alpha}, \alpha>1$, $l(x)=\alpha\log x$. We have
$\gamma=0$ and $\varepsilon_{1}(x)=\alpha/l(x)$ in (\ref{theta1}) so that the
restriction in $(CFG)$ is $\theta>1$.
\end{remark}
\noindent\textbf{A simple sufficient condition.} If we have, for some
$\zeta>2$%
\begin{equation}
\mathbb{P}\left(  X>x\right)  \leqslant\frac{1}{\exp(l(x))^{\zeta}},\quad
x\in\left(  m,+\infty\right)  ,\label{tractable}%
\end{equation}
then $\psi_{X}(x)\geqslant\zeta l(x)$ so that $(CFG)$ holds with arbitrarily
large $\theta$.\\
We use the following consequences of $(CFG)$.
Integrating $(CFG)$ yields 
\begin{equation}\label{eq:lpsi}
\psi_{X}\circ l^{-1}(x)\geqslant2x+2\theta\log x+K,\quad x\geqslant l(\tau_1),
\end{equation}
where the integrating constant $K$ does not matter and may change from line to line.
This also implies%
\begin{equation}
\psi_{X}(x)\geqslant2l(x)+2\theta\log l(x)+K,\quad x\geqslant \tau_1,\label{CFG1ab}%
\end{equation}
and, more importantly for our needs, inverting  \eqref{eq:lpsi} we obtain%
\begin{equation}\label{eq:psil}
l\circ\psi_{X}^{-1}(x)\leqslant\frac{x}{2}-\theta\log x+K, \quad x\geqslant \tau_1.\end{equation}
Now, \eqref{eq:lpsi} gives%
\[
\mathbb{P}\left(  \rho(X)>x\right)  =\mathbb{P}\left(  l(X)>\log x\right)
=\exp\left(  -\psi_{X}\circ l^{-1}(\log x)\right)  \leqslant\frac{K}%
{x^{2}(\log x)^{2\theta}}%
\]
and  since $\theta>1$ we have
\begin{equation}
\int_{m}^{+\infty}\sqrt{\mathbb{P}\left(  \rho(X)>x\right)  }dx<+\infty
.\label{hrhoX}%
\end{equation}
\begin{remark}
This is  the same kind of condition (3.4)  in  \cite{Ledoux17}
that ensures  the convergence of $W_1(\mathbb F_n,F)$ at rate $\sqrt n$. So
it turns out that (\ref{hrhoX}) is almost a minimal assumption in proving Theorem
\ref{mainth}. This is clearly confirmed at Lemmas \ref{variance} and \ref{variancebis}
establishing that the asymptotic variance of $\sqrt n \left(W_{c}(\F_{n},\G_{n})-W_{c}(F,G)\right)$ is finite.
\end{remark}
\begin{example}

For an over-exponential cost $c_{\gamma}^{+}$ from (\ref{cbeta}), $\gamma>1$, $(CFG)$ is satisfied if  $\mathbb{P}\left(  X>x\right)
\leqslant\exp(-2x^{\gamma}-\delta\log x)$ with $\delta>4(1-\gamma)$. For the Wasserstein cost
$c_{\alpha}$ from (\ref{calpha}), $\alpha>1$, consider a Pareto law, $\psi
_{X}(x)=\beta\log x$. Then $(CFG)$ reads $\alpha x/\beta<x/2-\theta\log x$,
and holds if $\beta>2\alpha$.
Gaussian laws are compatible without
restriction with any cost less than $\rho(x)=\exp(ax^{\gamma})$, $\gamma<2$,
$a>0$. In the case $\gamma=2$ the variance of $X$ has to be less than $a/4$
for $(CFG)$ to hold, and $G$ may be any Gaussian law different from $F$ with
smaller variance or same variance but smaller expectation.\smallskip
\end{example}

\section{Statement of the results}\label{sec:main}

\subsection{Consistency }
$W_c(\F_n,\G_n)$ is a consistent estimator of $W_c(F,G)$:
\begin{theorem}\label{theo:Consistance}
Assume that the good cost $c(x,y)$ is continuous, $F,G$ are strictly increasing and
$0\leqslant c(x,y)\leqslant V(x)+V(y)$ with $V$ a strictly increasing function such
that $\mathbb{E} V(X) <+\infty$ and $\mathbb{E}
V(Y)  <+\infty$. Then 
\[
\lim_{n\rightarrow\infty}W_c(\F_n,\G_n) =W_c(F,G)<+\infty \quad a.s.
\]

\end{theorem}

\subsection{A central limit theorem}

\begin{definition}
\label{defhyp}We say that conditions $(C)$, $(FG)$ and $(CFG)$ hold if they
hold for $(c(x,y),X,Y)$ as stated above and also for $(c(-x,-y),-X,-Y)$ with
possibly different functions $\rho$, $l$, $\psi$ and $F$ again
denoting the heavier tail.
\end{definition}

\noindent This means that the left hand tail of $F$ and $G$ should be reversed
from $\mathbb{R}_{-}$ to $\mathbb{R}_{+}$ and obey our set of conditions and,
if $G$ has the heavier tail the couples $(F,X)$ and $(G,Y)$ are simply
exchanged in $(FG)$ and $(CFG)$. \\

\noindent Define%
\[
\Pi(u,v)=\mathbb{P}\left(  X\leqslant F^{-1}(u),Y\leqslant G^{-1}(v)\right)  
\]
then the covariance matrix%
\begin{equation}
\Sigma(u,v)=\left(
\begin{array}
[c]{cc}%
\frac{\min(u,v)-uv}{h_{X}(u)h_{X}(v)} & \frac{\Pi(u,v)-uv}{h_{X}(u)h_{Y}(v)}\\
\frac{\Pi(v,u)-uv}{h_{X}(v)h_{Y}(u)} & \frac{\min(u,v)-uv}{h_{Y}(v)h_{Y}(u)}%
\end{array}
\right)  \label{sigmauv}%
\end{equation}
and the gradient%
\begin{equation}
\nabla(u)=\left(  \frac{\partial}{\partial x}c(F^{-1}(u),G^{-1}(u)),\frac
{\partial}{\partial y}c(F^{-1}(u),G^{-1}(u))\right)  .\label{deltau}%
\end{equation}
Our main result is the weak convergence of the empirical  Wasserstein distance
of (\ref{Lambdan}) toward an explicit Gaussian law $\cal N$.

\begin{theorem}
\label{mainth}If $(C)$, $(FG)$ and $(CFG)$ hold then%
\[
\sqrt{n}\left(  W_{c}(\F_{n},\mathbb G_{n})-W_{c}(F,G)\right)
\rightarrow_{law}\mathcal{N}(0,\sigma^{2}(\Pi,c))
\]
with%
\begin{equation}
\sigma^{2}(\Pi,c)=%
{\displaystyle\int\nolimits_{0}^{1}}
{\displaystyle\int\nolimits_{0}^{1}}
\nabla(u)\Sigma(u,v)\nabla(v)dudv<+\infty.\label{delta2}%
\end{equation}
Moreover for any real sequence $\varepsilon_{n}\rightarrow0$ then%
\[
W_{c,n}(\mathbb F_{n},\mathbb G_{n})=\int_{\varepsilon_{n}}^{1-\varepsilon_{n}}c(\mathbb F_{n}%
^{-1}(u),\mathbb G_{n}^{-1}(u))du,
\]
also satisfies%
\[
\sqrt{n}\left(  W_{c,n}(\mathbb F_{n},\mathbb G_{n})-W_{c}(F,G)\right)  \rightarrow
_{law}\mathcal{N}(0,\sigma^{2}(\Pi,c)).
\]
\end{theorem}
The result for the trimmed version $W_{c,n}$ easilly follows from the proof for $W_{c}$. Likewise slight changes in the proof of Theorem \ref{mainth} yields
\begin{theorem}
 If $(C)$, $(FG)$ and $(CFG)$ hold then%
\begin{align*}
&\sqrt{n}\left(  W_{c}(\F_{n}, G)-W_{c}(F,G)\right)
\rightarrow_{law}\mathcal{N}(0,\sigma_x^{2}(F,c)),\\
&\sqrt{n}\left(  W_{c}(F, \G_{n})-W_{c}(F,G)\right)
\rightarrow_{law}\mathcal{N}(0,\sigma_y^{2}(G,c)),
\end{align*}
with%
\begin{align*}
&\sigma_x^{2}(F,c)=%
{\displaystyle\int\nolimits_{0}^{1}}
{\displaystyle\int\nolimits_{0}^{1}}
\frac{\partial}{\partial x}c(F^{-1}(u),G^{-1}(u))\frac{\min(u,v)-uv}{h_{X}(u)h_{X}(v)}dudv<+\infty,\\
&\sigma^{2}_y(G,c)=%
{\displaystyle\int\nolimits_{0}^{1}}
{\displaystyle\int\nolimits_{0}^{1}}
\frac{\partial}{\partial y}c(F^{-1}(u),G^{-1}(u))\frac{\min(u,v)-uv}{h_{Y}(u)h_{Y}(v)}dudv<+\infty.
\end{align*}

\end{theorem}
In the particular case of the square Wasserstein distance and two independent samples we have
\begin{corollary} Assume that the two samples are independent, $(FG)$ holds and $\displaystyle\mathbb P(X>x)\leqslant \frac{1}{x^{4+\varepsilon}}$ with $\varepsilon>0$. Then  
$$\sqrt n (\frac{1}{n}\sum_{i=1}^n (X_{(i)}-Y_{(i)})^2- W_2^2(F,G))$$ weakly converges toward a centered Gaussian random variable  with variance
\begin{eqnarray*}
 \begin{split} &\sigma^2(F,G)=4 \int_0^1\int_0^1  \left(\frac{u\wedge v-uv}{f(F^{-1}(u)) f(F^{-1}(v))}+\frac{u\wedge v-uv}{g(G^{-1}(u)) g(G^{-1}(v))}\right)\\ &\times (F^{-1}(u)-G^{-1}(u)) (F^{-1}(v)-G^{-1}(v))dudv. \end{split} 
\end{eqnarray*}

\end{corollary}
For numerical application the following result could be useful.
\begin{corollary} Consider a family of c.d.f. $F_{a,b}(x)=F(\frac{x-b}{a})$, $a>~0, b\in~\mathbb R$. Assume that $F$ is symetric with variance 1, and denote $V_4=var(X^2)$ where $X$ has c.d.f. $F$. Then  it comes
\begin{eqnarray*}
\sigma^2(F_{a,b},F_{a',b'}))=4(a^2+a'^2)\left((b-b')^2+\frac{V_4}{4}(a-a')^2\right).
\end{eqnarray*}
\end{corollary}
As a consequence for two distinct Gaussian laws ${\cal N}(\nu,\zeta^2)$ and ${\cal N}(\mu,\xi^2)$ we obtain
$\sigma^2=4(\zeta^2+ \xi^2)(\nu-\mu)^2+2(\zeta^2+ \xi^2)(\zeta- \xi)^2 $as in Theorem 2.2 in \cite{rippl}.

\bigskip

We now go back  to our main result. It is easy to extend Theorem \ref{mainth}   to probability distributions supported by intervals.

\begin{theorem} Let $F$ and $G$ be supported by intervals. Assume that $(FG)$, $(C)$ and $(CFG)$ hold. If  the most  lightly tailed law  is compactly supported  $(FG4)$ is discarded. Then the conclusion of Theorem \ref{mainth} holds true.
\end{theorem}
\section{Conclusion}\label{sec:concl}
In this paper we have proved a CLT for the natural estimator of a wide class of probability distributions and Wasserstein costs. Our results concern a couple of samples having the same size but being possibly dependent, provided the marginal distributions are distinct enough. Thus it remains to handle three main problems. First, the case $F=G$ for which the speed of weak convergence could be different from the usual $\sqrt{n}$ and the limiting law could be non Gaussian.. Second the case of $W_1$ with $F\ne G$ and $F=G$. We will hopefully achieve these two studies in a forthcoming paper. The third problem is to extend our results to samples of different sizes, which  will impose to assume independence between them. 
\section{Proofs}\label{sec:proof}
\subsection{Proof of Theorem \ref{theo:Consistance}}
First we have
\[
0\leqslant W_{c}(F,G)\leqslant\int_{0}^{1}\left(  V(F^{-1}(u))+V(G^{-1}(u))\right)
du=\mathbb{E}  V(X)  +\mathbb{E}  V(Y) .
\]
Since $F$ is
strictly increasing by Glivenko-Cantelli's theorem the almost sure  convergence
$F_{n}^{-1}(u)\rightarrow F^{-1}(u)$ holds for any $u\in\left[  0,1\right]  $. Given
any $0<\alpha<\beta<1$, applying Dini's theorem to the increasing functions
$F_{n}^{-1}$ further yields%
\[
\lim_{n\rightarrow\infty}\sup_{u\in(\alpha,\beta)}\left\vert F_{n}%
^{-1}(u)-F^{-1}(u)\right\vert =\lim_{n\rightarrow\infty}\sup_{u\in
(\alpha,\beta)}\left\vert G_{n}^{-1}(u)-G^{-1}(u)\right\vert =0\  a.s.
\] It follows, by continuity of $c(x,y)$,%
\[
\lim_{n\rightarrow\infty}\int_{\alpha}^{\beta}\left\vert c(F_{n}^{-1}%
(u),G_{n}^{-1}(u))-c(F^{-1}(u),G^{-1}(u))\right\vert du=0\quad a.s.
\]
It remains to study%
\[
\frac{1}{n}%
{\displaystyle\sum\limits_{i=\left[  \beta n\right]  }^{n}}
c(X_{(i)},Y_{(i)})\leqslant\frac{1}{n}%
{\displaystyle\sum\limits_{i=\left[  \beta n\right]  }^{n}}
V(X_{(i)})+\frac{1}{n}%
{\displaystyle\sum\limits_{i=\left[  \beta n\right]  }^{n}}
V(Y_{(i)})
\]
since the lower quantiles sums can be handled similarily. Let $\beta^{-}%
<\beta$ and consider the random variable $Z_{i}=V(X_{i})$ and $\tilde Z _{i}%
=1_{Z_{i}\geqslant F_{Z}^{-1}(\beta^{-})}Z_{i}$ where $F_{Z}$ is the
$c.d.f.$ of $Z=V(X)$. 
Since ${\cal E}_X(\beta^-)=\mathbb{E} \tilde Z   \leqslant
\mathbb{E} V(X)  <+\infty$ it holds%
\[
\lim_{n\rightarrow+\infty}\frac{1}{n}%
{\displaystyle\sum\limits_{i=1}^{n}}
\tilde Z _{i}={\cal E}_X(\beta^-)\quad a.s.
\]
with ${\cal E}_X(\beta^-)\rightarrow0$ as $\beta^{-}\rightarrow1$.
Since  $F_{Z}$ is strictly increasing and the empirical
quantile of order $\beta$ of $Z_{1},...,Z_{n}$ is $V(F_{n}^{-1}(\beta))$  we get%
\[
\lim_{n\rightarrow+\infty}V(X_{(\left[  \beta n\right]  )})=F_{Z}^{-1}%
(\beta)>F_{Z}^{-1}(\beta^{-})\quad a.s.
\]
Therefore, with probability one we ultimately have%
\[
\frac{1}{n}
{\displaystyle\sum\limits_{i=\left[  \beta n\right]  }^{n}}
V(X_{(i)})\leqslant\frac{1}{n}%
{\displaystyle\sum\limits_{i=1}^{n}}
\tilde Z _{i}<2{\cal E}_X(\beta^-).
\]
To conclude we introduce two increasing sequences $\beta_{k}^{-}<\beta_{k}<1$
such that $\beta_{k}^{-}\rightarrow1$ as $k\rightarrow+\infty$ and consider
the associated ${\cal E}_X(\beta_k^-)=\int_{\beta_{k}^{-}}^{1}V(F^{-1}%
(u))\rightarrow0$ and ${\cal E}_Y(\beta_k^-)=\int_{\beta_{k}^{-}}^{1}V(G^{-1}%
(u))\rightarrow0$. Almost surely for  all $k$ simultaneously, using $G^{-1}(u)\leqslant F^{-1}(u)$ for $u$ large enough, we have%
\begin{align*}
&\lim_{n\rightarrow\infty}\int_{\alpha}^{\beta_{k}}\left\vert c(F_{n}%
^{-1}(u),G_{n}^{-1}(u))-c(F^{-1}(u),G^{-1}(u))\right\vert du =0\\
&\underset{n\rightarrow+\infty}{\lim\sup}\int_{\beta_{k}}^{1}c(F_{n}%
^{-1}(u),G_{n}^{-1}(u))du  \leqslant 2{\cal E}_X(\beta_k^-)+2{\cal E}_Y(\beta_k^-)\\
&\int_{\beta_{k}}^{1}c(F^{-1}(u),G^{-1}(u))du\leqslant{\cal E}_X(\beta_k^-)+{\cal E}_Y(\beta_k^-)%
\end{align*}
This
proves that $W_{c}(\F_{n},\mathbb G_{n})\rightarrow W_{c}(F,G)$ almost surely.

\subsection{Proof of Theorem \ref{mainth}}
The proof of Theorem  \ref{mainth} is organised as follows.\\
In Section \ref{sec:varlim} we prove  \eqref{delta2}.  Section \ref{sec:weak} is dedicated to the proof of the weak convergence of $\sqrt{n}\left(  W_{c}(\F_{n},\mathbb G_{n})-W_{c}(F,G)\right)$. Thanks to  definition \ref{defhyp} we only deal with the upper part of the integral. For that purpose we split the  interval $(1/2,1)$ into four parts, $(1/2,F(M))$, $(F(M), 1-h_n/n)$, $( 1-h_n/n,1-k_n/n)$, $( 1-k_n/n,1)$, where $F(M),h_n,k_n$ will be specified further on. The first integral is the main term and the other ones will be proved to be small. 
We study the integral over  $( 1-k_n/n,1)$ in 
 Step 1 of Section \ref{sec:weak},  the one over $( 1-h_n/n,1-k_n/n)$ in 
  Step 2 of Section \ref{sec:weak} and the one over $(F(M), 1-h_n/n)$ in 
Step 3 of Section \ref{sec:weak}. Finally, we deal with the main part in Step 4 of Section \ref{sec:weak}.

\subsubsection{The limiting variance.}\label{sec:varlim}

In this section we establish that $(C)$, $(FG)$ and $(CFG)$ imply that
$\sigma^{2}(\Pi,c)<+\infty$ in (\ref{delta2}). The covariance matrix
$\Sigma(u,v)$ and the gradient $\nabla(u)$ are defined at (\ref{sigmauv}) and
(\ref{deltau}). It is sufficient to study the right hand tails, corresponding
to the upper domain of integration $\left[  1/2,1\right]  ^{2}$. As a matter
of fact, this implies the same for $\left[  0,1/2\right]  ^{2}$ according to
Definition \ref{defhyp}, then similar arguments hold for mixing both tails
through $\left[  1/2,1\right]  \times\left[  0,1/2\right]  $ and $\left[
0,1/2\right]  \times\left[  1/2,1\right]  $ by separating the variables
exactly as we show below. Hence by cutting $\left[  1/2,1\right]  =\left[
1/2,\overline{u}\right]  \cup\left[  \overline{u},1\right]  $ into mid
quantiles and extremes we are reduced to control $\nabla(u)\Sigma
(u,v)\nabla(v)$ on $\left[  \overline{u},1\right]  \times\left[  \overline
{u},1\right]  $ then on $\left[  1/2,\overline{u}\right]  \times\left[
1/2,1\right]  $. The forthcoming two lemmas are then enough to conclude that
(\ref{delta2}) is true under $(C) $, $(FG)$ and $(CFG)$.

\begin{lemma}
\label{variance} Under $(C2)$, $(FG1)$, $(FG4)$ and $(CFG)$ we have, for any
$\overline{u}>F(m)$,%
\[
\sigma^{2}(\overline{u})=%
{\displaystyle\int\nolimits_{\overline{u}}^{1}}
{\displaystyle\int\nolimits_{\overline{u}}^{1}}
\nabla(u)\Sigma(u,v)\nabla(v)dudv<+\infty.
\]

\end{lemma}

\begin{proof}
By $(C2)$ we have, for $x\geqslant y\geqslant m$,%
\[
\frac{\partial}{\partial x}c\left(  x,y\right)  =-\frac{\partial}{\partial
y}c\left(  x,y\right)  =\frac{\partial}{\partial x}\rho(x-y)=l^{\prime
}(x-y)\rho(x-y)=\rho^{\prime}(x-y).
\]
By  $(FG4)$  it holds $
F^{-1}(u)\geqslant\tau(u)=F^{-1}(u)-G^{-1}(u)\geqslant\tau_{0}>0$ for $u>F(m)$. 
Thus, for $u\in\left[  \overline{u},1\right]  $, $
\nabla(u)=\left(  \rho^{\prime}\circ\tau(u),-\rho^{\prime}\circ\tau(u)\right)
.$
Let us split $\sigma^{2}(\overline{u})$ into 
\begin{align*}
A_{1}  & =%
{\displaystyle\int\nolimits_{\overline{u}}^{1}}
{\displaystyle\int\nolimits_{\overline{u}}^{1}}
\rho^{\prime}\circ\tau(u)\frac{\min(u,v)-uv}{h_{X}(u)h_{X}(v)}\rho^{\prime
}\circ\tau(v)dudv\\
A_{2}  & =-%
{\displaystyle\int\nolimits_{\overline{u}}^{1}}
{\displaystyle\int\nolimits_{\overline{u}}^{1}}
\rho^{\prime}\circ\tau(u)\frac{\Pi(v,u)-uv}{h_{X}(v)h_{Y}(u)}\rho^{\prime
}\circ\tau(v)dudv\\
A_{3}  & =-%
{\displaystyle\int\nolimits_{\overline{u}}^{1}}
{\displaystyle\int\nolimits_{\overline{u}}^{1}}
\rho^{\prime}\circ\tau(u)\frac{\Pi(u,v)-uv}{h_{X}(u)h_{Y}(v)}\rho^{\prime
}\circ\tau(v)dudv\\
A_{4}  & =%
{\displaystyle\int\nolimits_{\overline{u}}^{1}}
{\displaystyle\int\nolimits_{\overline{u}}^{1}}
\rho^{\prime}\circ\tau(u)\frac{\min(u,v)-uv}{h_{Y}(v)h_{Y}(u)}\rho^{\prime
}\circ\tau(v)dudv.
\end{align*}
Observe that if $0<u<v<1$ then%
\[
0\leqslant\frac{\min(u,v)-uv}{\sqrt{1-u}\sqrt{1-v}}=u\sqrt{\frac{1-v}{1-u}%
}\leqslant1
\]
so that we always have $0\leqslant\min(u,v)-uv\leqslant\sqrt{1-u}\sqrt{1-v}$
and we get 
\[
\left\vert A_{1}\right\vert \leqslant\left(
{\displaystyle\int\nolimits_{\overline{u}}^{1}}
\rho^{\prime}\circ\tau(u)\frac{\sqrt{1-u}}{h_{X}(u)}du\right)  ^{2}%
,\quad\left\vert A_{4}\right\vert \leqslant\left(
{\displaystyle\int\nolimits_{\overline{u}}^{1}}
\rho^{\prime}\circ\tau(u)\frac{\sqrt{1-u}}{h_{Y}(u)}du\right)  ^{2}.
\]
Consider the bound of $\left\vert A_{1}\right\vert $ first. By $(C2)$, 
$\rho^{\prime}$ is $\mathcal{C}_{1}(m,+\infty)$ and positive. Now, as
$u\rightarrow1$, $\tau(u)\geqslant\tau_{0}>0$ is either unbounded or bounded.
In both cases we have
\[
0<\rho^{\prime}(\tau(u))\leqslant\max\left(  \rho^{\prime}\circ F^{-1}%
(u),\sup_{\tau_{0}<x\leqslant l_{2}}\rho^{\prime}(x)\right)  \leqslant
k_{1}\rho^{\prime}\circ F^{-1}(u)
\]
for $k_{1}\geqslant1$ since by Proposition \ref{convex} the increasing
function $\rho$ is convex on $\left(  l_{2},+\infty\right)  $ under $(C2)$.
Observe that $\rho$ is also invertible, so that $\rho(X)$ has quantile
function, density function and density quantile function respectively given
by
\begin{equation}
F_{\rho(X)}^{-1}=\rho\circ F^{-1},\quad f_{\rho(X)}=\frac{f\circ\rho^{-1}%
}{\rho^{\prime}\circ\rho^{-1}},\quad h_{\rho(X)}=f_{\rho(X)}\circ F_{\rho
(X)}^{-1}=\frac{h_{X}}{\rho^{\prime}\circ F^{-1}}.\label{Frho}%
\end{equation}
Recalling that $(CFG)$ implies (\ref{hrhoX}), the change of variable
$x=\rho\circ F^{-1}(u)$ yields%
\begin{align*}
\frac{1}{k_{1}}\int_{\overline{u}}^{1}\rho^{\prime}\circ\tau(u)\frac
{\sqrt{1-u}}{h_{X}(u)}du  & \leqslant\int_{F(m)}^{1}\rho^{\prime}\circ
F^{-1}(u)\frac{\sqrt{1-u}}{h_{X}(u)}du\\
& =\int_{F(m)}^{1}\frac{\sqrt{1-u}}{h_{\rho(X)}(u)}du\\
& =\int_{\rho(m)}^{+\infty}\sqrt{\mathbb{P}\left(  \rho(X)>x\right)
}dx<+\infty.
\end{align*}
Having proved that $\left\vert A_{1}\right\vert <+\infty$ let us next study
the upper bound of $\left\vert A_{4}\right\vert $. Under $(C2)$ and
(\ref{theta1}) we have, for some $\varepsilon_{1}(x)\rightarrow\gamma$,%
\[
\rho^{\prime}(x)=l^{\prime}(x)\rho(x)=\varepsilon_{1}(x)\frac{l(x)}{x}%
\rho(x)\leqslant(1+\gamma)\frac{l(x)^{\theta_{1}^{\prime}}}{x}\rho(x)
\]
where $\theta_{1}^{\prime}\in(\theta_{1},\theta-1)$ if $\gamma=0$, and
$\theta_{1}^{\prime}=1$ if $\gamma>0$. It then follows from the change of
variable $u=G(x)$ that, by setting $\phi=G^{-1}\circ F=\psi_{Y}^{-1}\circ\psi_{X}$,
\begin{align}
&
{\displaystyle\int\nolimits_{\overline{u}}^{1}}
\rho^{\prime}\circ\tau(u)\frac{\sqrt{1-u}}{h_{Y}(u)}du\nonumber\\
& \leqslant(1+\gamma)%
{\displaystyle\int\nolimits_{G^{-1}(\overline{u})}^{+\infty}}
\frac{(l\circ\phi^{-1}(x))^{\theta_{1}^{\prime}}}{\phi^{-1}(x)}\rho\circ
\phi^{-1}(x)\sqrt{\mathbb{P}\left(  Y>x\right)  }dx.\label{rhohy}%
\end{align}
Now, by $(FG4)$ we have%
\[
x\leqslant\phi^{-1}(x)=F^{-1}\circ G(x)=\psi_{X}^{-1}\circ\psi_{Y}(x)=\psi
_{X}^{-1}\left(  \log\left(  \frac{1}{\mathbb{P}\left(  Y>x\right)  }\right)
\right)
\]
thus by \eqref{eq:psil}  we have
\[
l\circ\phi^{-1}(x)\leqslant\frac{1}{2}\log\left(  \frac{1}{\mathbb{P}\left(
Y>x\right)  }\right)  -\theta\log\log\left(  \frac{1}{\mathbb{P}\left(
Y>x\right)  }\right)+K .
\]
We can bound (\ref{rhohy}) from above by%
\begin{align*}
& (1+\gamma)
{\displaystyle\int\nolimits_{\phi(m)}^{+\infty}}
\frac{(l\circ\phi^{-1}(x))^{\theta_{1}^{\prime}}}{\phi^{-1}(x)}\exp\left(
l\circ\phi^{-1}(x)\right)  \sqrt{\mathbb{P}\left(  Y>x\right)  }dx\\
& \leqslant K
{\displaystyle\int\nolimits_{\phi(m)}^{+\infty}}
\frac{\left(  \psi_{Y}(x)\right)  ^{\theta_{1}^{\prime}-\theta}}{\psi_{X}^{-1}\circ\psi_{Y}(x)}dx\\
& \leqslant K
{\displaystyle\int\nolimits_{\phi(m)}^{+\infty}}
\frac{1}{x\left(  \psi_{Y}(x)\right)  ^{\theta
-\theta_{1}^{\prime}}}dx\\
& \leqslant K
{\displaystyle\int\nolimits_{\phi(m)}^{+\infty}}
\frac{1}{x\left(  l(x)\right)  ^{\theta-\theta_{1}^{\prime}}}dx.
\end{align*}
The last inequality comes from $\psi_Y(x)>\psi_X(x)$ by ($FG4$). If $\gamma>0$ then
$\theta-\theta_{1}^{\prime}=\theta-1>0$ and $l(x)>x^{\gamma/2}$ hence the
bounding integral is finite. If $\gamma=0$ then $l(x)\geqslant\log x$ by
(\ref{L1}) and having enforced $\theta-\theta_{1}>\theta-\theta_{1}^{\prime
}>1$ also makes the above integral finite.
We have shown that $\left\vert
A_{4}\right\vert <+\infty$. It remains to bound $A_{2}=A_{3}$. Since $F$ and
$G$ are continuous it holds%
\begin{align*}
\Pi(u,v)  & \leqslant\min\left(  \mathbb{P}\left(  X\leqslant F^{-1}%
(u)\right)  ,\mathbb{P}(Y\leqslant G^{-1}(v))\right)  =\min\left(  u,v\right)
\\
\Pi(u,v)  & \geqslant\mathbb{P}\left(  X\leqslant F^{-1}(u)\right)
+\mathbb{P}(Y\leqslant G^{-1}(v))-1=u+v-1
\end{align*}
and thus%
\begin{align*}
\Pi(u,v)-uv  & \leqslant\min(u,v)-uv\leqslant\sqrt{1-u}\sqrt{1-v}\\
\Pi(u,v)-uv  & \geqslant u+v-1-uv=-(1-u)(1-v)
\end{align*}
which proves that $
\left\vert \Pi(u,v)-uv\right\vert \leqslant\sqrt{1-u}\sqrt{1-v}$.
Hence $A_{2}=A_{3}$ satisfies%
\begin{align*}
\left\vert A_{2}\right\vert  & \leqslant%
{\displaystyle\int\nolimits_{\overline{u}}^{1}}
\rho^{\prime}\circ\tau(v)\frac{\sqrt{1-v}}{h_{X}(v)}dv%
\displaystyle\int\nolimits_{\overline{u}}^{1}
\rho^{\prime}\circ\tau(u)\frac{\sqrt{1-u}}{h_{Y}(u)}du\\
& \leqslant k_{1}^{2}\int_{F(m)}^{1}\rho^{\prime}\circ F^{-1}(v)\frac
{\sqrt{1-v}}{h_{X}(v)}du\int_{F(m)}^{1}\rho^{\prime}\circ F^{-1}(u)\frac
{\sqrt{1-u}}{h_{Y}(u)}du
\end{align*}
where these integrals are already proved to be finite. Finally $\sigma
^{2}(\overline{u})=A_{1}+A_{2}+A_{3}+A_{4}<+\infty$.
\end{proof}

\begin{lemma}
\label{variancebis}Under $(C1)$, $(C2)$, $(FG1)$, $(FG4)$\ and $(CFG)$ we have, for any
$\overline{u}>F(m)$,%
\begin{align*}
&\sigma_{-}^{2}(\overline{u})=%
{\displaystyle\int\nolimits_{1/2}^{\overline{u}}}
{\displaystyle\int\nolimits_{1/2}^{1}}
\nabla(u)\Sigma(u,v)\nabla(v)dudv<+\infty,\\&
 \sigma_{+}^{2}(\overline{u})={\displaystyle\int\nolimits_{\overline{u}}^{1}}
{\displaystyle\int\nolimits_{1/2}^{\overline{u}}}
\nabla(u)\Sigma(u,v)\nabla(v)dudv<+\infty.
\end{align*}

\end{lemma}

\begin{proof}
Since $F^{-1}$ and $G^{-1}$ are bounded on $\left[  1/2,\overline{u}\right]  $
we have, by $(C1)$, that $\nabla(u)$ exists and is bounded on $\left[
1/2,\overline{u}\right]  $. Likewise $(FG1)$ ensures that $h_{X}$ and $h_{Y}$
are bounded on $\left[  1/2,\overline{u}\right]  $ hence $\Sigma(u,v)$ is
bounded on $\left[  1/2,\overline{u}\right]  ^{2}$. As a consequence,%
\[
 A_{0}=%
{\displaystyle\int\nolimits_{1/2}^{\overline{u}}}
{\displaystyle\int\nolimits_{1/2}^{\overline{u}}}
\nabla(u)\Sigma(u,v)\nabla(v)dudv, \quad \left\vert A_{0}\right\vert <+\infty.
\]
By $(C2)$ we have $\nabla(u)=\left(  \rho^{\prime}\circ
\tau(u),-\rho^{\prime}\circ\tau(u)\right)  $ on $\left[  \overline
{u},1\right]  $, thus%
\begin{align*}
A_{01}  & =%
{\displaystyle\int\nolimits_{1/2}^{\overline{u}}}
{\displaystyle\int\nolimits_{\overline{u}}^{1}}
\frac{\partial}{\partial x}c\left(  F^{-1}(u),G^{-1}(u)\right)  \frac
{\min(u,v)-uv}{h_{X}(u)h_{X}(v)}\rho^{\prime}\circ\tau(u)dudv\\
A_{02}  & =-
{\displaystyle\int\nolimits_{1/2}^{\overline{u}}}
{\displaystyle\int\nolimits_{\overline{u}}^{1}}
\frac{\partial}{\partial y}c\left(  F^{-1}(u),G^{-1}(u)\right)  \frac
{\Pi(v,u)-uv}{h_{X}(v)h_{Y}(u)}\rho^{\prime}\circ\tau(u)dudv\\
A_{03}  & =-%
{\displaystyle\int\nolimits_{1/2}^{\overline{u}}}
{\displaystyle\int\nolimits_{\overline{u}}^{1}}
\frac{\partial}{\partial x}c\left(  F^{-1}(u),G^{-1}(u)\right)  \frac
{\Pi(u,v)-uv}{h_{X}(u)h_{Y}(v)}\rho^{\prime}\circ\tau(u)dudv\\
A_{04}  & =%
{\displaystyle\int\nolimits_{1/2}^{\overline{u}}}
{\displaystyle\int\nolimits_{\overline{u}}^{1}}
\frac{\partial}{\partial y}c\left(  F^{-1}(u),G^{-1}(u)\right)  \frac
{\min(u,v)-uv}{h_{Y}(v)h_{Y}(u)}\rho^{\prime}\circ\tau(u)dudv.
\end{align*}
Along the same arguments as in Lemma \ref{variance} we have%
\[
\left\vert A_{01}\right\vert \leqslant I_{X}J_{X},\quad\left\vert
A_{02}\right\vert \leqslant I_{Y}J_{X},\quad\left\vert A_{03}\right\vert
\leqslant I_{X}J_{Y},\quad\left\vert A_{04}\right\vert \leqslant I_{Y}J_{Y},
\]
where, by the previous boundedness argument on $\left[  1/2,\overline
{u}\right]  $,%
\begin{align*}
I_{X}  & =\left(
{\displaystyle\int\nolimits_{1/2}^{\overline{u}}}
\left\vert \frac{\partial}{\partial x}c\left(  F^{-1}(u),G^{-1}(u)\right)
\right\vert \frac{\sqrt{1-u}}{h_{X}(u)}\right)  <+\infty\\
I_{Y}  & =\left(
{\displaystyle\int\nolimits_{1/2}^{\overline{u}}}
\left\vert \frac{\partial}{\partial y}c\left(  F^{-1}(u),G^{-1}(u)\right)
\right\vert \frac{\sqrt{1-u}}{h_{Y}(u)}\right)  <+\infty
\end{align*}
and by $(CFG)$, (\ref{eq:lpsi}), (\ref{CFG1ab}) and (\ref{hrhoX}) on $\left[
\overline{u},1\right]  $,%
\begin{align*}
\frac{J_{X}}{k_{1}}  & =\left(
{\displaystyle\int\nolimits_{\overline{u}}^{1}}
\rho^{\prime}\circ F^{-1}(v)\frac{\sqrt{1-v}}{h_{X}(v)}dv\right)  <+\infty\\
\frac{J_{Y}}{k_{1}}  & =\left(
{\displaystyle\int\nolimits_{\overline{u}}^{1}}
\rho^{\prime}\circ F^{-1}(v)\frac{\sqrt{1-v}}{h_{Y}(v)}dv\right)  <+\infty.
\end{align*}
Therefore $\sigma_{-}^{2}(\overline{u})=A_{0}+A_{01}+A_{02}+A_{03}%
+A_{04}<+\infty$. In the same way the result holds for  $\sigma_{+}^{2}(\overline{u})$.
\end{proof}

\subsubsection{Proof of the weak convergence}\label{sec:weak}

\begin{bf}Step1: Extreme Values\end{bf}\\
In this first step we show that the
contribution of extremes is negligible despite the rate $\sqrt{n}$. 
Without
information on joint laws of extreme values we treat separately the upper tail
of the integrals $W_{c}(\F_{n},\G_{n})$ and $W_{c}(F,G)$. Indeed the latter in
not a centering of the former at the very end of tails so that the empirical
quantile processes cannot help.\smallskip

\noindent Let $K_{n}$ be a positive increasing sequence such that
\begin{equation}
K_{n}\rightarrow+\infty,\quad\frac{K_{n}}{\log\log n}\rightarrow0.\label{Kn}%
\end{equation}
Define%
\begin{equation}
k_{n}= \frac{\sqrt{n}}{K_{n}\exp\left(  l\circ\psi_{X}^{-1}(\log
n+K_{n})\right)  } .\label{kn}%
\end{equation}
Under $(C2)$ and $(FG1)$ we have $l\circ\psi_{X}^{-1}(x)\rightarrow+\infty$ as
$x\rightarrow+\infty$ thus $k_{n}=o\left(  \sqrt{n}/K_{n}\right)  $. Moreover,
by \eqref{eq:psil}  and (\ref{Kn}) for any
$\theta^{\prime}\in\left(  1,\theta\right)  $ and all $n$ large enough it
holds%
\begin{equation}
k_{n}\geqslant\frac{c}{K_{n}}\exp\left(  -\frac{K_{n}}{2}+\theta\log(\log
n+K_{n})\right)  >(\log n)^{\theta^{\prime}}.\label{knbis}%
\end{equation}
Hence we have $\displaystyle {k_{n}/\log\log n\rightarrow+\infty}$ and $k_n/\sqrt n\rightarrow0$. Let us define
\begin{align*}
D_{n}&=%
{\displaystyle\int\nolimits_{1-k_{n}/n}^{1}}
c\left(  F^{-1}(u),G^{-1}(u)\right)  du,\\
S_{n}&=%
{\displaystyle\int\nolimits_{1-k_{n}/n}^{1}}
c\left(  \F_{n}^{-1}(u),\G_{n}^{-1}(u)\right)  du=\frac{1}{n}%
{\displaystyle\sum\limits_{i=n-[k_{n}]}^{n}}
c\left(  X_{(i)},Y_{(i)}\right)  .
\end{align*}

\begin{lemma}\label{lem:DS}
\begin{enumerate}
\item\label{cvDn}Assume that $(C2),$ $(FG1)$, $(FG4)$ and $(CFG)$ hold. Then%
\[
\sqrt{n}D_{n}\rightarrow0.
\]
\item \label{cvSn}Under $(C2)$ and $(CFG)$, we have%
\[
\sqrt{n}S_{n}\rightarrow0\quad\text{in probability.}%
\]
\end{enumerate}
\end{lemma}

\begin{proof}
\begin{enumerate}
\item
By $C_2$ and $FG_4$ we have%
\[
D_{n}=%
{\displaystyle\int\nolimits_{1-k_{n}/n}^{1}}
\rho\left(  F^{-1}(u)-G^{-1}(u)\right)  du\leqslant%
{\displaystyle\int\nolimits_{1-k_{n}/n}^{1}}
w(u)du
\]
where%
\[
w(u)=\exp\left(  l\circ F^{-1}(u)\right)  =\exp\left(  l\circ\psi_{X}%
^{-1}\left(  \log\left(  \frac{1}{1-u}\right)  \right)  \right)  .
\]
Under $(CFG)$, for $\theta>1$ it holds, by (\ref{eq:psil}),%
\[
l\circ F^{-1}(u)\leqslant\frac{1}{2}\log\left(  \frac{1}{1-u}\right)
-\theta\log\log\left(  \frac{1}{1-u}\right)+K
\]
thus, as $n\rightarrow+\infty$,%
\[%
{\displaystyle\int\nolimits_{1-k_{n}/n}^{1}}
w(u)du\leqslant\left[  -\frac{K\sqrt{1-u}}{\left(  \log\left(  1/(1-u)\right)
\right)  ^{\theta}}\right]  _{1-k_{n}/n}^{1}=\frac{K\sqrt{k_{n}/n}}{\left(
\log(n/k_{n})\right)  ^{\theta}}\rightarrow0
\]
so that $w(u)$ is integrable on $\left(  \overline{u},1\right)  $. 
By $(CFG)$ 
$\varphi=\psi_{X}\circ l^{-1}$ satisfies
\[
\left(  \varphi^{-1}\right)  ^{\prime}(x)=\frac{1}{\varphi^{\prime}%
\circ\varphi^{-1}(x)}\leqslant\frac{1}{2+2\theta/\varphi^{-1}(x)}%
\]
and for $x$ large enough,%
\begin{equation}
  \left(\varphi^{-1}\right)  ^{\prime}(x)=\left(  l\circ\psi_{X}^{-1}\right)  ^{\prime}(x)\leqslant\frac{1}%
{2+2\theta/(x/2-\theta\log x+K)}<\frac{1}{2}-\frac{\theta}{x}.\label{CFG1ac}%
\end{equation}
We then have%
\[
\left(  -\left(  1-u\right)  w(u)\right)  ^{\prime}=w(u)\left(  1-\left(
l\circ\psi_{X}^{-1}\right)  ^{\prime}\left(  \log\left(  \frac{1}{1-u}\right)
\right)  \right)  >\frac{w(u)}{2}%
\]
which gives%
\[%
{\displaystyle\int\nolimits_{1-k_{n}/n}^{1}}
w(u)du\leqslant2\left[  -\left(  1-u\right)  w(u)\right]  _{1-k_{n}/n}%
^{1}\leqslant\frac{2k_{n}}{n}w\left(  1-\frac{k_{n}}{n}\right)  ,
\]
since $\lim_{u\to 1}\left(  1-u\right)  w(u)=0$.
Recalling (\ref{kn}) it follows that for  $n$ large
enough,%
\begin{align*}
\sqrt{n}D_{n}  & \leqslant\frac{2k_{n}}{\sqrt{n}}\exp\left(  l\circ\psi
_{X}^{-1}\left(  \log\left(  \frac{n}{k_{n}}\right)  \right)  \right) \\
& \leqslant\frac{2}{K_{n}}\exp\left(  l\circ\psi_{X}^{-1}\left(  \log\left(
\frac{n}{k_{n}}\right)  \right)  -l\circ\psi_{X}^{-1}(\log n+K_{n})\right)  .
\end{align*}
By (\ref{Kn}), (\ref{kn}) and (\ref{eq:psil}) with $\theta>1$ we
get%
\begin{align*}
\log\left(  \frac{n}{k_{n}}\right)   & \sim\frac{\log n}{2}+\log K_{n}+l\circ
\psi_{X}^{-1}(\log n+K_{n})\\
& \leqslant\log n+\frac{K_{n}}{2}+\log K_{n}-\theta\log(\log n+K_{n}%
)+K%
\end{align*}
hence $\sqrt{n}D_{n}\leqslant2/K_{n}\rightarrow0$ as $n\rightarrow+\infty$
since $l\circ\psi_{X}^{-1}$ is increasing.\smallskip

\item Next we control $S_n$ the stochastic sum of extreme values. %
Fix $\delta>0$ and consider the events%
\[
A_{n}=\left\{  \sqrt{n}S_{n}\geqslant4\delta\right\}  ,\quad B_{n,X}=\left\{
X_{(n-[k_{n}])}> m\right\}  ,\quad B_{n,Y}=\left\{  Y_{(n-[k_{n}]%
)}> m\right\}  .
\]
We have $$\mathbb{P}\left(  A_{n}\right)   \leqslant\mathbb{P}\left(  A_{n}\cap
B_{n,X}\cap B_{n,Y}\right)  +\mathbb{P}\left(  B_{n,X}^{c}\right)
+\mathbb{P}\left(  B_{n,X}^{c}\right) .$$

Since $F$ and $G$ are strictly increasing we obviously have, for $\xi>0$ and
$u_{0}=F(m+\xi)$, as $n\rightarrow+\infty$,%
\begin{align*}
\mathbb{P}\left(  B_{n,X}^{c}\right)   & =\mathbb{P}\left(  \F_{n}^{-1}\left(
1-\frac{k_{n}}{n}\right)  <m\right) \\
& \leqslant\mathbb{P}\left(  \F_{n}^{-1}(u_{0})<F^{-1}(u_{0})-\xi\right)
\rightarrow0
\end{align*}
and likewise, $\mathbb{P}\left(  B_{n,Y}^{c}\right)  \rightarrow0$. By
(\ref{crhosup}) we can write, under $B_{n,X}\cap B_{n,Y}$,%
\begin{align*}
\sqrt{n}S_{n}  & \leqslant\frac{1}{\sqrt{n}}%
{\displaystyle\sum\limits_{i=n-[k_{n}]}^{n}}
\left(  \rho\left(  X_{(i)}\right)  +\rho\left(  Y_{(i)}\right)  \right) \\
& \leqslant\frac{k_{n}+1}{\sqrt{n}}\left(  \rho\left(  X_{(n)}\right)
+\rho\left(  Y_{(n)}\right)  \right)
\end{align*}
hence $\mathbb{P}\left(  A_{n}\cap B_{n,X}\cap B_{n,Y}\right)  \leqslant
\mathbb{P}\left(  C_{n,X}\right)  +\mathbb{P}\left(  C_{n,Y}\right)  $ where%
\[
C_{n,X}=\left\{  \rho\left(  X_{(n)}\right)  \geqslant\delta\frac{\sqrt{n}%
}{k_{n}}\right\}  ,\quad C_{n,Y}=\left\{  \rho\left(  Y_{(n)}\right)
\geqslant\delta\frac{\sqrt{n}}{k_{n}}\right\}  .
\]
Now we have, by $(FG4)$ and since $X_{1},...,X_{n}$ are independent,%
\[
\mathbb{P}\left(  C_{n,Y}\right)  \leqslant\mathbb{P}\left(  C_{n,X}\right)
=1-\left(  1-\mathbb{P}\left(  \rho(X)>\delta\frac{\sqrt{n}}{k_{n}}\right)
\right)  ^{n}%
\]
then combining $\rho^{-1}(x)=l^{-1}(\log x)$ with (\ref{kn}) gives,
\[
\mathbb{P}\left(  \rho(X)>\delta\frac{\sqrt{n}}{k_{n}}\right)  =\exp\left(
-\psi_{X}\circ l^{-1}\left(  \log\delta+l\circ\psi_{X}^{-1}\left(  \log n+\log
K_{n}\right)  +\log K_{n}\right)  \right)
\]
Now by ($CFG$) $\psi_{X}\circ l^{-1}$ is
increasing.  As soon as $\log K_{n}>\left\vert \log\delta\right\vert $ we get  %
\[
\mathbb{P}\left(  \rho(X)>\delta\frac{\sqrt{n}}{k_{n}}\right)  \leqslant
\exp\left(  -\psi_{X}\circ l^{-1}\left(  l\circ\psi_{X}^{-1}\left(  \log
n+\log K_{n}\right)  \right)  \right)  =\frac{1}{n K_{n}},%
\]
which yields%
\[
\mathbb{P}\left(  C_{n,X}\right)  \leqslant1-\exp\left(  -\frac{K}{ K_{n}%
}\right)  \rightarrow0.
\]
We conclude that%
\begin{align*}
\mathbb{P}\left(  A_{n}\right)   & \leqslant\mathbb{P}\left(  A_{n}\cap
B_{n,X}\cap B_{n,Y}\right)  +\mathbb{P}\left(  B_{n,X}^{c}\right)
+\mathbb{P}\left(  B_{n,X}^{c}\right) \\
& \leqslant\mathbb{P}\left(  C_{n,X}\right)  +\mathbb{P}\left(  C_{n,X}%
\right)  +\mathbb{P}\left(  B_{n,X}^{c}\right)  +\mathbb{P}\left(  B_{n,X}%
^{c}\right)
\end{align*}
satisfies $\mathbb{P}\left(  A_{n}\right)  \rightarrow0$.
\end{enumerate}
\end{proof}

\noindent\begin{bf}Step2: Centered high order quantiles\end{bf}\\

This section ends the part of the proof of Theorem \ref{mainth} devoted to the
secondary order.  We split the arguments
into the three lemmas below. Remind that $k_{n}$ is defined at (\ref{kn}). Let
introduce%
\begin{equation}
h_{n}=n^{\beta},\quad\beta\in\left(  \frac{1}{2},1\right)  ,\quad
I_{n}=\left[  1-\frac{h_{n}}{n},1-\frac{k_{n}}{n}\right]  ,\label{In}%
\end{equation}
and define the centered random integral of non extreme tail quantiles to be%
\[
T_{n}=%
{\displaystyle\int\nolimits_{1-h_{n}/n}^{1-k_{n}/n}}
\left(  c\left(  \F_{n}^{-1}(u),\G_{n}^{-1}(u)\right)  -c\left(  F^{-1}%
(u),G^{-1}(u)\right)  \right)  du.
\]

\begin{lemma}\label{lem:T}
\label{cvTn}Under $(C2)$, $(FG)$ and $(CFG)$ we have%
\[
\underset{n\rightarrow+\infty}{\lim}\sqrt{n}T_{n}=0\quad\text{a.s.}%
\]

\end{lemma}
 The proof of this lemma is based on the two following lemmas whose proof are postponed in the appendix.
In order to bound $T_{n}$ we first evaluate the quantile empirical processes%
\begin{equation}
\beta_{n}^{X}(u)=\sqrt{n}(\F_{n}^{-1}(u)-F^{-1}(u)),\quad\beta_{n}^{Y}%
(u)=\sqrt{n}(\G_{n}^{-1}(u)-G^{-1}(u)).\label{betan}%
\end{equation}

\begin{lemma}
\label{hongrois}Define $\Delta_{n}=\left[  \overline{u},1-k_{n}/n\right]  $.
Under $(FG1)$ and $(FG2)$ we have%
\[
\underset{n\rightarrow+\infty}{\lim\sup}\sup_{u\in\Delta_{n}}\frac{\left\vert
\beta_{n}(u)\right\vert h(u)}{\sqrt{\left(  1-u\right)  \log\log n}}%
\leqslant4\quad\text{a.s.}%
\]
where $\left(  \beta_{n},h\right)  =(\beta_{n}^{X},h_{X})$ or $\left(
\beta_{n},h\right)  =(\beta_{n}^{Y},h_{Y})$.
\end{lemma}

\noindent In the next key lemma we have to carefully check that the conditions
given at Proposition \ref{C3rho} are almost surely met on $I_{n}\subset  \Delta_{n}$.
For $u\in I_{n}$ and $n\geqslant3$ define%
\begin{equation}
\varepsilon_{n}(u)=\varepsilon_{n}^{X}(u)-\varepsilon_{n}^{Y}(u),\quad
\varepsilon_{n}^{X}(u)=\frac{\beta_{n}^{X}(u)}{\sqrt{n}},\quad\varepsilon
_{n}^{Y}(u)=\frac{\beta_{n}^{Y}(u)}{\sqrt{n}}.\label{epsilon}%
\end{equation}

\begin{lemma}\label{lem:truc2}
\label{DL}Assume that $(C2)$, $(FG)$ and $(CFG)$ hold. Then there exists $K_{2}>0$ such that%
\[
\underset{n\rightarrow+\infty}{\lim\sup}\sup_{u\in I_{n}}\frac{\left\vert
c\left(  \F_{n}^{-1}(u),\G_{n}^{-1}(u)\right)  -c\left(  F^{-1}(u),G^{-1}%
(u)\right)  \right\vert }{\rho^{\prime}\circ F^{-1}(u)\text{ }\left\vert
\varepsilon_{n}(u)\right\vert }\leqslant K_{2}\quad a.s.
\]

\end{lemma}

\begin{proof} {\bf of Lemma \ref{lem:T}}\\
Remind notation from (\ref{kn}), (\ref{In}) and (\ref{epsilon}). By Lemma
\ref{DL} it holds, with probability one, for all $n$ large enough%
\begin{align*}
\left\vert T_{n}\right\vert   %
 \leqslant K%
{\displaystyle\int\nolimits_{1-h_{n}/n}^{1-k_{n}/n}}
\rho^{\prime}\circ F^{-1}(u)\left\vert \varepsilon_{n}(u)\right\vert du.
\end{align*}
We  proceed as in the proof of Lemma
\ref{variance} where similar integrable functions show up, however they have
now to be integrated to sharply evaluate $\sqrt{n}\left\vert T_{n}\right\vert
$. From Lemma \ref{hongrois} it follows, with probability one, that for all
$n$ large and all $u\in I_{n}\subset\Delta_{n}$,%
\begin{equation}\label{eq:eps}
\left\vert \varepsilon_{n}(u)\right\vert \leqslant\left\vert \frac{\beta
_{n}^{X}(u)}{\sqrt{n}}\right\vert +\left\vert \frac{\beta_{n}^{Y}(u)}{\sqrt
{n}}\right\vert \leqslant 5\sqrt{\frac{\log\log n}{n}}\left(  \frac
{\sqrt{1-u}}{h_{X}(u)}+\frac{\sqrt{1-u}}{h_{Y}(u)}\right)  .
\end{equation}
We then compute separately the following two integrals%
\[
\sqrt{n}\left\vert T_{n}\right\vert \leqslant5K\sqrt{\log\log
n}\left(
{\displaystyle\int\nolimits_{1-h_{n}/n}^{1-k_{n}/n}}
t_{X}(u)du+%
{\displaystyle\int\nolimits_{1-h_{n}/n}^{1-k_{n}/n}}
t_{Y}(u)du\right)
\]
where, for $Z=X,Y$ we write $\displaystyle
t_{Z}(u)=\rho^{\prime}\circ F^{-1}(u)\frac{\sqrt{1-u}}{h_{Z}(u)}
$.\\
\textbf{First integral.} Since $\rho$ is convex by Proposition \ref{convex} we
can use (\ref{Frho}) as in the proof of Lemma \ref{variance} to justify the
change of variable $u=F\circ\rho^{-1}(x)$ then apply (\ref{psy}) to $\rho
^{-1}(x)=l^{-1}(\log x)$ and rewrite the first integral as%
\begin{align*}%
{\displaystyle\int\nolimits_{1-h_{n}/n}^{1-k_{n}/n}}
t_{X}(u)du  & =%
{\displaystyle\int\nolimits_{1-h_{n}/n}^{1-k_{n}/n}}
\frac{\sqrt{1-u}}{h_{\rho(X)}(u)}du\\
& =%
{\displaystyle\int\nolimits_{b(n/h_{n})}^{b(n/k_{n})}}
\sqrt{\mathbb{P}\left(  \rho(X)>x\right)  }dx\\
& =%
{\displaystyle\int\nolimits_{b(n/h_{n})}^{b(n/k_{n})}}
\exp\left(  -\frac{1}{2}\psi_{X}\circ l^{-1}(\log x)\right)  dx
\end{align*}
where, by $(CFG)$ reformulated into (\ref{eq:psil}),%
\begin{equation}
b(x)=\rho\circ F^{-1}\left(  1-\frac{1}{x}\right)  =\exp\left(  l\circ\psi
_{X}^{-1}(\log x)\right)  \leqslant\frac{K\sqrt{x}}{(\log x)^{\theta}%
}.\label{bx}%
\end{equation}
Equation (\ref{hrhoX}) justifies that $t_{X}$ is integrable since
$\theta>1$ and, by (\ref{eq:lpsi}),%
\[
\exp\left(  -\frac{1}{2}\psi_{X}\circ l^{-1}(\log x)\right)  \leqslant\frac
{K}{x(\log x)^{\theta}}.
\]
Now observe that $\varphi=\psi_{X}\circ l^{-1}$ satisfies $\varphi^{\prime
}=\left(  \psi_{X}^{\prime}/l^{\prime}\right)  \circ l^{-1}$ and $(CFG)$
reads%
\[
\varphi^{\prime}(x)\geqslant2+\frac{2\theta}{x},\quad x>l(\tau_1),
\]
so that we have, for all $x>b(n/h_{n})>l(\tau_1)$,%
\begin{align*}
\left(  -x\exp\left(  -\frac{1}{2}\varphi(\log x)\right)  \right)  ^{\prime}
& =\left(  \frac{1}{2}\varphi^{\prime}(\log x)-1\right)  \exp\left(  -\frac
{1}{2}\varphi(\log x)\right) \\
& \geqslant\frac{\theta}{\log x}\exp\left(  -\frac{1}{2}\varphi(\log
x)\right)  .
\end{align*}
Therefore it holds, thanks to the upper bound (\ref{bx}) and since $b(x)$ is
increasing,%
\begin{align*}%
{\displaystyle\int\nolimits_{1-h_{n}/n}^{1-k_{n}/n}}
t_{X}(u)du  & \leqslant\frac{\log b(n/k_{n})}{\theta}%
{\displaystyle\int\nolimits_{b(n/h_{n})}^{b(n/k_{n})}}
\frac{\theta}{\log x}\exp\left(  -\frac{1}{2}\psi_{X}\circ l^{-1}(\log
x)\right)  dx\\
& \leqslant\frac{\log b(n)}{\theta}\left[  -x\exp\left(  -\frac{1}{2}\psi
_{X}\circ l^{-1}(\log x)\right)  \right]  _{b(n/h_{n})}^{b(n/k_{n})}\\
& \leqslant\frac{K\log n}{\theta}\frac{b(n/h_{n})}{\sqrt{n/h_{n}}}\\
& =\frac{K}{\theta(1-\beta)^{\theta}(\log n)^{\theta-1}}%
\end{align*}
since $h_n=n^\beta$. This proves that
\[
\underset{n\rightarrow+\infty}{\lim}\sqrt{\log\log n}%
{\displaystyle\int\nolimits_{1-h_{n}/n}^{1-k_{n}/n}}
t_{X}(u)du=0.
\]
\noindent\textbf{Second integral.} Next consider%
\[
J_{n}=%
{\displaystyle\int\nolimits_{1-h_{n}/n}^{1-k_{n}/n}}
t_{Y}(u)du=%
{\displaystyle\int\nolimits_{1-h_{n}/n}^{1-k_{n}/n}}
l^{\prime}\circ F^{-1}(u)\frac{\sqrt{1-u}}{h_{Y}(u)}\rho\circ F^{-1}(u)du.
\]
By (\ref{phiprime}) and (\ref{L'}), under $(C2)$ we have $l^{\prime
}(x)=\varepsilon_{1}(x)l(x)/x$ with $\varepsilon_{1}(x)\rightarrow\gamma$ as
$x\rightarrow+\infty$. If $\gamma=0$ the rate of $\varepsilon_{1}(x)$ is given
by (\ref{theta1}) and we pick $\theta_{1}^{\prime}\in(\theta_{1},\theta-1)$.
If $\gamma>0$ let $\theta_{1}^{\prime}=1$. Recall that $\phi^{-1}=F^{-1}\circ
G=\psi_{X}^{-1}\circ\psi_{Y}$. Start with%
\begin{align*}
J_{n}  & \leqslant (1+\gamma)%
{\displaystyle\int\nolimits_{1-h_{n}/n}^{1-k_{n}/n}}
\frac{(l\circ F^{-1}(u))^{\theta_{1}^{\prime}}}{F^{-1}(u)}\frac{\sqrt{1-u}%
}{h_{Y}(u)}\rho\circ F^{-1}(u)du\\
& =(1+\gamma)%
{\displaystyle\int\nolimits_{G^{-1}(1-h_{n}/n)}^{G^{-1}(1-k_{n}/n)}}
\frac{(l\circ\phi^{-1}(x))^{\theta_{1}^{\prime}}}{\phi^{-1}(x)}\exp\left(
l\circ\phi^{-1}(x)\right)  \sqrt{\mathbb{P}\left(  Y>x\right)  }dx.
\end{align*}
Observe that $(CFG)$ and  (\ref{eq:psil}) imply%
\[
l\circ\phi^{-1}(x)=l\circ\psi_{X}^{-1}\circ\psi_{Y}\left(  x\right)
\leqslant\frac{\psi_{Y}\left(  x\right)  }{2}-\theta\log\psi_{Y}\left(
x\right) +K .
\]
Since $\psi_{X}^{-1}\circ\psi_{Y}(x)\geqslant x$ by $(FG4)$ and $\psi
_{Y}^{\prime}(x)\geqslant K/x$ by $(FG5)$ it readily
follows, for $\theta-\theta_{1}^{\prime}>1$ and $K>0$,%
\begin{align*}
J_{n}  & \leqslant (1+\gamma)%
{\displaystyle\int\nolimits_{G^{-1}(1-h_{n}/n)}^{G^{-1}(1-k_{n}/n)}}
\frac{\left(  \psi_{Y}(x)\right)  ^{\theta_{1}^{\prime}-\theta}}{\psi_{X}%
^{-1}\circ\psi_{Y}(x)}dx\\
& \leqslant K%
{\displaystyle\int\nolimits_{G^{-1}(1-h_{n}/n)}^{G^{-1}(1-k_{n}/n)}}
\frac{\psi_{Y}^{\prime}(x)}{\left(  \psi_{Y}(x)\right)  ^{\theta-\theta
_{1}^{\prime}}}dx\\
& =K\left[  \frac{-1}{\left(  \psi_{Y}(x)\right)  ^{\theta-\theta
_{1}^{\prime}-1}}\right]  _{\psi_{Y}^{-1}(\log(n/h_{n}))}^{\psi_{Y}^{-1}%
(\log(n/k_{n}))}\\
& \leqslant\frac{K}{\left(  (1-\beta)\log n\right)  ^{\theta-\theta
_{1}^{\prime}-1}}%
\end{align*}
therefore%
\[
\underset{n\rightarrow+\infty}{\lim}\sqrt{\log\log n}%
{\displaystyle\int\nolimits_{1-h_{n}/n}^{1-k_{n}/n}}
t_{Y}(u)du=0.
\]
As a conclusion, the almost sure upper bound of $\sqrt{n}\left\vert
T_{n}\right\vert $ tends to zero.
\end{proof}

\noindent\begin{bf}Step 3: Upper middle order quantiles\end{bf}\\

At (\ref{In}) we have defined $h_{n}=n^{\beta}$ with $\beta\in\left(
1/2,1\right)  $ to be chosen. Let us introduce%
\begin{equation}
\label{Imn}
I_{M,n}=\left(  F(M),1-\frac{h_{n}}{n}\right)  ,\quad M>m\text{.}%
\end{equation}
Since $F(M)>F(m)=\overline{u}$ and  (\ref{ctorho}) in Section \ref{sectruc2} holds we have by $(C2)$%
\begin{align*}
U_{M,n}  & =%
{\displaystyle\int\nolimits_{F(M)}^{1-h_{n}/n}}
\left(  c\left( \mathbb F_{n}^{-1}(u),\mathbb G_{n}^{-1}(u)\right)  -c\left(  F^{-1}%
(u),G^{-1}(u)\right)  \right)  du\\
& =%
{\displaystyle\int\nolimits_{F(M)}^{1-h_{n}/n}}
\rho\left(  \left\vert \tau(u)+\varepsilon_{n}(u)\right\vert \right)
-\rho\left(  \tau(u)\right)  du
\end{align*}
where $\varepsilon_{n}(u)$ is as in (\ref{epsilon}). 
In order to control the last integral, we expand $\rho$ and
make use of a distribution free Brownian approximation of the joint quantile processes. 
\begin{lemma}
\label{cvUMn}Assume $(C2)$, $(FG)$ and $(CFG)$. For any $\varepsilon
>0$ and  $\lambda>0$ we can find $M>m$ such that, for all $n$ large enough,%
\[
\mathbb{P}\left(  \sqrt{n}\left\vert U_{M,n}\right\vert >\lambda\right)
<\varepsilon.
\]

\end{lemma}

\begin{proof}
\begin{enumerate}
\item Under $(C2)$ we have $l^{\prime}(x)=\varepsilon_{1}(x)l(x)/x$
where $\varepsilon_{1}(x)\rightarrow\gamma$ as $x\rightarrow+\infty$ thus
$\varepsilon_{1}$ is bounded on $\left(  M,+\infty\right)  $. Moreover,
$(CFG)$ ensures that
\[
l\circ\psi_{Y}^{-1}(x)\leqslant l\circ\psi_{X}^{-1}(x)<x
\]
whereas (\ref{CFG1ab}) and (\ref{L1}) entails that $\psi_{X}(x)>2l(x)\geqslant
2\log x$ thus
\[
F^{-1}(u)=\psi_{X}^{-1}\left(  \log\left(  \frac{1}{1-u}\right)  \right)
<\frac{1}{\sqrt{1-u}}%
\]
for all $u\in I_{M,n}$ and $x\in F(I_{M,n})$. Under $(FG4)$ we have
$\tau(u)=F^{-1}(u)-G^{-1}(u)\geqslant\tau_{0}$ for $u\in I_{M,n}$. Hence
by choosing $M>m$ and $K>0$ sufficiently large, (\ref{eq:eps}) and
$(FG3)$ imply that it almost surely eventually holds%
\begin{align*}
& \sup_{u\in I_{M,n}}  \varepsilon
_{1}\circ\tau(u)\frac{l\circ\tau(u)}{\tau(u)}\left\vert\varepsilon_{n}(u)\right\vert\\
& \leqslant K\sup_{u\in I_{M,n}}l\circ F^{-1}%
(u)(H_{X}(u)F^{-1}(u)+H_{Y}(u)G^{-1}(u))\sqrt{\frac{\log\log n}{n(1-u)}}\\
& \leqslant K\sqrt{\log\log n}\sup_{u\in I_{M,n}}\frac{l\circ\psi_{X}%
^{-1}(\log(1/(1-u))}{\sqrt{n(1-u)}}F^{-1}(u)\\
& \leqslant K\sqrt{n\log\log n}\sup_{u\in I_{M,n}}\frac{\log(1/(1-u)}%
{n(1-u)}\\
& \leqslant K\frac{\log n}{h_{n}}\sqrt{n\log\log n}%
\end{align*}
which vanishes since $\beta>1/2$ in (\ref{In}). We have shown that%
\begin{equation}
\underset{n\rightarrow+\infty}{\lim}\sup_{u\in I_{M,n}}\left\vert
\varepsilon_{n}(u)\right\vert l^{\prime}\circ\tau(u)=0\quad a.s.\label{DLuIMn}%
\end{equation}
\item  By  (\ref{DLuIMn}), the second part of Proposition
\ref{C3rho} can be applied for all large $n$. It says that
\[
\rho\left(  \left\vert \tau(u)+\varepsilon_{n}(u)\right\vert \right)
-\rho\left(  \tau(u)\right)  =k_{0}(\tau(u),\varepsilon_{n}(u))\rho^{\prime
}\circ\tau(u)\varepsilon_{n}(u)
\]
where, by (\ref{DLexplicit}),%
\[
\lim_{\delta_{0}\rightarrow0}\sup_{\tau(u)>\tau_{0}}\sup_{\left\vert
\varepsilon_{n}(u)\right\vert l^{\prime}\circ\tau(u)\leqslant\delta_{0}%
}\left\vert k_{0}(\tau(u),\varepsilon_{n}(u))-1\right\vert =0
\]
which can be reformulated through (\ref{DLuIMn}) into $k_{1}(u)=k_{0}%
(\tau(u),\varepsilon_{n}(u))$ and%
\begin{equation}
\underset{n\rightarrow+\infty}{\lim}\sup_{u\in I_{M,n}}\left\vert
k_{1}(u)-1\right\vert =0\quad a.s.\label{DLk1}%
\end{equation}
Thus, given any $\vartheta\in\left(  0,1\right)  $ the random function
$k_{1}(u)$ is such that $k_{1}(u)\in\left(  1-\vartheta,1+\vartheta\right)  $
for all $u\in I_{M,n}$ and%
\[
\sqrt{n}U_{M,n}=%
{\displaystyle\int\nolimits_{F(M)}^{1-h_{n}/n}}
k_{1}(u)\rho^{\prime}\circ\tau(u)\left(  \beta_{n}^{X}(u)+\beta_{n}%
^{Y}(u)\right)  du.
\]
From now on we work on the  probability space  of Theorem \ref{approx}.  This allows us to write%
\[
\sqrt{n}U_{M,n}=%
{\displaystyle\int\nolimits_{F(M)}^{1-h_{n}/n}}
k_{1}(u)\rho^{\prime}\circ\tau(u)\left(  \frac{B_{n}^{X}(u)+Z_{n}%
^{X}(u)}{h_{X}(u)}+\frac{B_{n}^{Y}(u)+Z_{n}^{Y}(u)}{h_{Y}(u)}\right)  du
\]
where $\left(  U_{M,n},B_{n}^{X},Z_{n}^{X},B_{n}^{Y},Z_{n}^{Y}%
,k_{1}\right)  $ are built together on $\Omega^{\ast}$ in such a way
that  for some small $\xi>0$ independent
of the law $\Pi$,%
\begin{equation}
\underset{n\rightarrow+\infty}{\lim}n^{\xi}\sup_{u\in I_{M,n}}\left\vert
Z_{n}^{X}(u)\right\vert =\underset{n\rightarrow+\infty}{\lim}n^{\xi}\sup_{u\in
I_{M,n}}\left\vert Z_{n}^{Y}(u)\right\vert =0\quad a.s.\label{Z}%
\end{equation} and  $B_{n}^{X}, B_{n}^{Y}$ are Brownian briges define at (\ref{ponts}).  Therefore  $k_{1}$ obeys (\ref{DLk1}).

Let set $\sqrt{n}\,U_{M,n}=N_{M,n}+R_{M,n}+S_{M,n}$
with 
\begin{align*}
N_{M,n}  & =%
{\displaystyle\int\nolimits_{F(M)}^{1-h_{n}/n}}
\rho^{\prime}\circ\tau(u)\left(  \frac{B_{n}^{X}(u)}{h_{X}(u)}+\frac{B_{n}%
^{Y}(u)}{h_{Y}(u)}\right)  du\\
R_{M,n}  & =%
{\displaystyle\int\nolimits_{F(M)}^{1-h_{n}/n}}
k_{1}(u)\rho^{\prime}\circ\tau(u)\left(  \frac{Z_{n}^{X}(u)}{h_{X}(u)}+\frac{Z_{n}%
^{Y}(u)}{h_{Y}(u)}\right)  du\\
S_{M,n}  & =%
{\displaystyle\int\nolimits_{F(M)}^{1-h_{n}/n}}
(k_{1}(u)-1)\rho^{\prime}\circ\tau(u)\left(  \frac{B_{n}^{X}(u)}{h_{X}(u)}+\frac{B_{n}%
^{Y}(u)}{h_{Y}(u)}\right)  du
\end{align*}

\item We first deal with $R_{M,n}$. Since $\rho^{\prime}(x)$ is increasing by Proposition
\ref{convex}, $(C2)$ implies $l^{\prime}(x)<Kl(x)/x$ with $K>\gamma$ and $(CFG)$ entails $l\circ\psi_{Y}^{-1}(x)\leqslant l\circ\psi_{X}^{-1}(x)\leqslant x/2-\theta\log
x$ by (\ref{eq:lpsi}) we readily have%
\begin{align*}
&
\left\vert{\displaystyle\int\nolimits_{F(M)}^{1-h_{n}/n}}
\frac{\rho^{\prime}\circ\tau(u)}{h_{X}(u)}Z_n^X(u)du\right\vert\\
& \leqslant \frac{K}{n^\xi}%
{\displaystyle\int\nolimits_{F(M)}^{1-h_{n}/n}}
\frac{l\circ\psi_{X}^{-1}(\log(1/(1-u))}{F^{-1}(u)h_{X}(u)}\exp\left(
l\circ\psi_{X}^{-1}(\log(1/(1-u))\right)  du\\
& \leqslant \frac{K}{n^\xi}%
{\displaystyle\int\nolimits_{F(M)}^{1-h_{n}/n}}
\frac{\log(1/(1-u))}{(\log(1/(1-u)))^{\theta}}\frac{H_{X}(u)}{(1-u)^{3/2}}du
\end{align*}
which is, by using $(FG3)$ and $\theta>1$ then choosing $\beta\in\left(
1-\xi,1\right)  $, less than%
\begin{equation*}
\frac{K}{n^\xi}%
{\displaystyle\int\nolimits_{F(M)}^{1-h_{n}/n}}
\frac{1}{(1-u)^{3/2}}du<Kn^{-\xi
/2}.\label{integraleZ}%
\end{equation*}
The same bound  holds for $h_{Y}$ since $F^{-1}>G^{-1}$ and%
\[%
\left\vert{\displaystyle\int\nolimits_{F(M)}^{1-h_{n}/n}}
\frac{\rho^{\prime}\circ\tau(u)}{h_{Y}(u)}Z_n^Y(u)du\right\vert\leqslant
\frac{K}{n^\xi}%
{\displaystyle\int\nolimits_{F(M)}^{1-h_{n}/n}}
\frac{G^{-1}(u)}{F^{-1}(u)}\frac{\log(1/(1-u))}{(\log(1/(1-u)))^{\theta}}%
\frac{H_{Y}(u)}{(1-u)^{3/2}}du.
\]
\smallskip
By (\ref{DLk1}), (\ref{Z}) and  the above bounds we have almost surely for $n$ large enough 
$\displaystyle\left\vert R_{M,n}\right\vert  \leqslant 2Kn^{-\xi/2}
\longrightarrow0$.
\item  As $N_{M,n}$ is the sum of two linear
functionals of Brownian bridges it is a mean zero Gaussian random variable
with variance%
\[
\sigma^{2}(M,n)=%
{\displaystyle\int\nolimits_{F(M)}^{1-h_{n}/n}}
{\displaystyle\int\nolimits_{F(M)}^{1-h_{n}/n}}
\rho^{\prime}\circ\tau(u)\rho^{\prime}\circ\tau(v)\Xi(u,v)dudv
\]
 where%
\begin{align*}
\Xi(u,v)  & =cov\left(  \frac{B_{n}^{X}(u)}{h_{X}(u)}+\frac{B_{n}^{Y}%
(u)}{h_{Y}(u)},\frac{B_{n}^{X}(v)}{h_{X}(v)}+\frac{B_{n}^{Y}(v)}{h_{Y}%
(v)}\right) \\
& =\frac{\min(u,v)-uv}{h_{X}(u)h_{X}(v)}+\frac{\Pi(v,u)-uv}{h_{X}(v)h_{Y}%
(u)}+\frac{\Pi(u,v)-uv}{h_{X}(u)h_{Y}(v)}+\frac{\min(u,v)-uv}{h_{Y}%
(v)h_{Y}(u)}.
\end{align*}
Therefore by Lemma \ref{variance} taken in $\overline{u}=F(M)$ we see that $\sigma^{2}(M,n)\rightarrow
\sigma^{2}(M)$ as $n\rightarrow
\infty$ and $\sigma^{2}(M)\rightarrow0$ as $M\rightarrow+\infty$.  On an other hand the random variable 
${\displaystyle\int\nolimits_{F(M)}^{1-h_{n}/n}}
\rho^{\prime}\circ\tau(u)\left(  \frac{B_{n}^{X}(u)}{h_{X}(u)}+\frac{B_{n}%
^{Y}(u)}{h_{Y}(u)}\right)  du$ is $a.s.$ finite. Thus ${\displaystyle\int\nolimits_{F(M)}^{1-h_{n}/n}}
\rho^{\prime}\circ\tau(u)\left\vert  \frac{B_{n}^{X}(u)}{h_{X}(u)}+\frac{B_{n}%
^{Y}(u)}{h_{Y}(u)}\right\vert  du$ is $a.s.$ finite. Hence
 $$\vert S_{M,n}\vert\leqslant\sup_{u\in I_{M,n}}\left\vert
k_{1}(u)-1\right\vert   {\displaystyle\int\nolimits_{F(M)}^{1-h_{n}/n}}
\rho^{\prime}\circ\tau(u)\left\vert  \frac{B_{n}^{X}(u)}{h_{X}(u)}+\frac{B_{n}%
^{Y}(u)}{h_{Y}(u)}\right\vert  du,$$
which $a.s.$ tend to $0$ when $n\to \infty$.\\
As a conclusion, for any $\varepsilon>0$ 
and $\lambda>0$ we can find $M=M(\varepsilon,\lambda)>m$ such that%
\begin{align*}
\mathbb{P}\left(  \sqrt{n}\left\vert U_{M,n}\right\vert >\lambda\right)
&\leqslant\mathbb{P}\left(  \left\vert N_{M,n}\right\vert
>\frac{\lambda}{3}\right)  +\mathbb{P}\left(  \left\vert R_{M,n}\right\vert >\frac{\lambda}{3}\right) +\mathbb{P}\left(  \left\vert S_{M,n}\right\vert >\frac{\lambda}{3}\right)\\
& \leqslant\frac{\sigma^{2}(M,n)}%
{(\lambda/3)^{2}}+\frac{\varepsilon}{3}+ \frac{\varepsilon}{3}<   \varepsilon,
\end{align*}
for all $n>n(\varepsilon,\lambda,M)$.
\end{enumerate}
\end{proof}

\noindent\begin{bf}Step 4: Centered middle order quantiles\end{bf}\\

 Define
\[
I_M=\left(  F(-M),F(M)\right)  ,\quad M>m,
\]
and consider the centered random integral%
\[
\M_{M,n}=%
{\displaystyle\int\nolimits_{I_M}}
\left(  c\left(  \F_{n}^{-1}(u),\G_{n}^{-1}(u)\right)  -c\left(  F^{-1}%
(u),G^{-1}(u)\right)  \right)  du.
\]
In order to conclude the proof of Theorem \ref{mainth} it remains to exploit
the Brownian approximation of the joint quantile processes $\beta_{n}^{X}$ and
$\beta_{n}^{Y}$ defined at (\ref{betan}) to accurately approximate $\sqrt
{n}\M_{M,n}$. Recalling (\ref{deltau}) let us write%
\[
\nabla_{x}(u)=\frac{\partial}{\partial x}c(F^{-1}(u),G^{-1}(u)),\quad
\nabla_{y}(u)=\frac{\partial}{\partial y}c(F^{-1}(u),G^{-1}(u))
\]
and
\[
\sqrt{n}\mathbb N_{M,n}=%
{\displaystyle\int\nolimits_{I_M}}
\left(  \nabla_{x}(u)\beta_{n}^{X}(u)+\nabla_{y}(u)\beta_{n}^{Y}(u)\right)
du.
\]

\begin{lemma}
\label{MnNn}Assume $(C)$, $(FG)$ and $(CFG)$. Then for any $\delta>0$, any
$\varepsilon>0$ and any $M>m^{\prime}>m$ there exists $n(\varepsilon
,\delta,M)$ such that for all $n>n(\varepsilon,\delta,M)$,%
\[
\mathbb{P}\left(  \left\vert \sqrt{n}\M_{M,n}-\sqrt{n}\mathbb N_{M,n}\right\vert
>\varepsilon\right)  \leqslant\delta.
\]

\end{lemma}

\begin{proof}
\begin{enumerate}
\item Under $(FG1)$, $h_{X}$ and $h_{Y}$ are away from $0$ on
$I_M$ and we write%
\[
\eta_{M}=\min\left(  \inf_{u\in I_M}h_{X}(u),\inf_{u\in I_M}%
h_{Y}(u)\right)>0  .
\]
We keep working on the probability space of Theorem \ref{approx}. In particular, since $I_M\subset\mathcal{I}_n$ we can aplly again Theorem  \ref{approx} and get the analogue of (\ref{Z})%
\begin{equation}
\mathbb{P}\left(  \sup_{u\in I_M}\left\vert\frac{ Z_{n}^{X}(u)}{h_{X}(u)}%
\right\vert >\frac{1}{n^{\xi}}\right)=o(1)%
,\mathbb{P}\left(  \sup_{u\in I_M}\left\vert \frac{Z_{n}%
^{Y}(u)}{h_{Y}(u)}\right\vert >\frac{1}{n^{\xi}}\right) =o(1)%
.\label{Zbis}%
\end{equation}
\smallskip
Introduce the event%
\[
A_{n}(M,C)=\left\{  \sup_{u\in I_M}\left\vert \F_{n}^{-1}(u)-F^{-1}%
(u)\right\vert +\left\vert \G_{n}^{-1}(u)-G^{-1}(u)\right\vert \leqslant
\frac{4C}{\sqrt{n}}\right\}  .
\]
By (\ref{Zbis}), for any $\delta>0$ one can find $C_{\delta}>0$ so large that,
for all $n$ large enough,%
\begin{align*}
& \mathbb{P}\left(  A_{n}(M,C_{\delta})^{c}\right) \\
& =\mathbb{P}\left(  \sup_{u\in I_M}\sqrt{n}\left\vert \F_{n}%
^{-1}(u)-F^{-1}(u)\right\vert +\sqrt{n}\left\vert \G_{n}^{-1}%
(u)-G^{-1}(u)\right\vert >4C_{\delta}\right) \\
& \leqslant\mathbb{P}\left(  \sup_{u\in I_M}\left\vert \frac
{B_{n}^{X}(u)}{h_{X}(u)}\right\vert >C_{\delta}\right)  +\mathbb{P}
\left(  \sup_{u\in I_M}\left\vert \frac{B_{n}^{Y}(u)}{h_{Y}%
(u)}\right\vert >C_{\delta}\right)  +o(1)\\
& \leqslant2\mathbb{P}\left(  \sup_{u\in I_M}\left\vert
B(u)\right\vert >\eta_{M}C_{\delta}\right)  +\frac{\delta}{2}\\
& \leqslant\delta
\end{align*}
where $B$ denotes a standard Brownian bridge. \\

\item  Since $F\neq G$ and $F,G$ are continuous, for any
$\tau_{1}\in\left(  0,\tau_{0}\right)  $ there exists an open interval
$I(\tau_{1})\subset I_M$ such that $\left\vert \tau(u)\right\vert
>\tau_{1}$ for $u\in I(\tau_{1})$, provided that $m>0$ is chosen large enough.
By taking $M>m$, by $(FG4)$ we further have $\tau_{M}=\sup_{u\in I_M%
}\left\vert \tau(u)\right\vert \geqslant\tau_{0}>\tau_{1}$. Thus
\begin{align*}
D_{M}^{+}(\tau_{1})  & =\left\{  u:\tau_{1}<\left\vert \tau(u)\right\vert
\leqslant\tau_{M}\right\}  \cap I_M\\
D_{M}^{-}(\tau_{1})  & =\left\{  u:\left\vert \tau(u)\right\vert \leqslant
\tau_{1}\right\}  \cap I_M%
\end{align*}
are such that $I(\tau_{1})\subset D_{M}^{+}(\tau_{1})%
\neq\emptyset$ and $D_{M}^{-}(\tau_{1})\subset I_M$ is possibly empty, and $D_{M}^{+}(\tau_{1})\cup D_{M}^{-}(\tau_{1})=I_M$. By
$(C3)$, for any $\left(x,y\right)$, $\left(
x^{\prime},y^{\prime}\right)  \in D_{m}(\tau)$,%
\[
\left\vert c\left(  x^{\prime},y^{\prime}\right)  -c\left(  x,y\right)
\right\vert \leqslant d(m,\tau)\left(  \left\vert x^{\prime}-x\right\vert
+\left\vert y^{\prime}-y\right\vert \right)
\]
with $d(m,\tau)\rightarrow0$ as $\tau\rightarrow0$ and $m$ is fixed. Observe
that $u\in D_{M}^{-}(\tau_{1})=D_{m}^{-}(\tau_{1})$ if, and only if, $\left(
F^{-1}(u),G^{-1}(u)\right)  \in D_{m}(\tau_{1})$. 
Let $\tau_{1}^{\prime}\in\left(
\tau_{1},\tau_{0}\right)  $ and $m^{\prime}\in\left(  m,M\right)  $.

 Now,
given $M$ and $C_{\delta}$, if $A_{n}(M,C_{\delta})$ is true for a large
enough $n$ then $\left(  \mathbb F_{n}^{-1}(u),\mathbb G_{n}^{-1}(u)\right)  \in D_{m^{\prime
}}(\tau_{1}^{\prime})$ whenever $\left(  F^{-1}(u), G^{-1}(u)\right)  \in
D_{m}(\tau_{1})\subset D_{m^{\prime}}(\tau_{1}^{\prime})$ and $u\in I_M$.
Thus, under the event $A_{n}(M,C_{\delta})$ it holds%
\begin{align*}
\sqrt{n}\M_{M,n}^{-}(\tau_{1})  & :=\sqrt{n}%
{\displaystyle\int\nolimits_{u\in D_{M}^{-}(\tau_{1})}}
\left\vert c\left(  \F_{n}^{-1}(u),\G_{n}^{-1}(u)\right)  -c\left(
F^{-1}(u),G^{-1}(u)\right)  \right\vert du\\
& \leqslant\sqrt{n}%
{\displaystyle\int\nolimits_{u\in D_{M}^{-}(\tau_{1})}}
d(m^{\prime},\tau_{1}^{\prime})\left(  \left\vert \F_{n}^{-1}(u)-F^{-1}%
(u)\right\vert +\left\vert \G_{n}^{-1}(u)-G^{-1}(u)\right\vert \right)  du\\
& \leqslant4C_{\delta}d(m^{\prime},\tau_{1}^{\prime}).
\end{align*}

\item The main term is%
\[
\sqrt{n}\M_{M,n}^{+}(\tau_{1}):=\sqrt{n}%
{\displaystyle\int\nolimits_{u\in D_{M}^{+}(\tau_{1})}}
\left(  c\left(  \F_{n}^{-1}(u),\G_{n}^{-1}(u)\right)  -c\left(  F^{-1}%
(u),G^{-1}(u)\right)  \right)  du.
\]
Under the event $A_{n}(M,C_{\delta})$ the Taylor expansion of $c(F^{-1}%
(u),G^{-1}(u))$ is justified on $D_{M}^{+}(\tau_{1})$, that is away from the
diagonal. 
As a matter of fact, under $(C1)$ we have, for $x,y$ in
$\left(  -M,M\right)  $ such that $\left\vert x-y\right\vert \geqslant\tau$,%
\[
\left\vert c\left(  x+\varepsilon_{x},y+\varepsilon_{y}\right)  -c\left(
x,y\right)  -\nabla_{x}(x,y)\varepsilon_{x}-\nabla_{y}(x,y)\varepsilon
_{y}\right\vert \leqslant\lambda(M,\tau)\Theta\left(  \left\vert
\varepsilon_{x}\right\vert +\left\vert \varepsilon_{y}\right\vert \right)  ,
\]
where $\Theta\left(  s\right)  /s\rightarrow0$ as $s\rightarrow0$ for $M$ and
$\tau_{1}$ fixed.
Then the expansion of $c(F^{-1}(u),G^{-1}(u))$ on $u\in D_{M}^{+}(\tau_{1})$ can be
writen as

 $$c\left( \F_{n}^{-1}(u),\G_{n}^{-1}(u)\right)-c(F^{-1}(u),G^{-1}(u)=\left(  \nabla_{x}(u)\beta_{n}^{X}(u)+\nabla_{y}(u)\beta_{n}^{Y}(u)\right)+{\cal R}_n(u).$$
We have
\begin{align*}
& \left\vert \sqrt{n}\M_{M,n}^{+}(\tau_{1})-%
{\displaystyle\int\nolimits_{u\in D_{M}^{+}(\tau_{1})}}
\left(  \nabla_{x}(u)\beta_{n}^{X}(u)+\nabla_{y}(u)\beta_{n}^{Y}(u)\right)
du\right\vert\\
&\leqslant \sqrt n \left\vert \int\nolimits_{u\in D_{M}^{+}(\tau_{1})}{\cal R}_n(u)du\right\vert\\
&\leqslant\lambda(M,\tau_{1})\sqrt{n}\Theta\left(  \frac{1}{\sqrt{n}}%
\sup_{u\in I_M}\left\vert \beta_{n}^{X}(u)\right\vert +\left\vert
\beta_{n}^{Y}(u)\right\vert \right)
\end{align*}
As $\vert \M_{M,n}-\M_{M,n}^{+}(\tau_{1})\vert \leqslant \M_{M,n}^{-}(\tau_{1})$, whenever $A_{n}(M,C_{\delta})$ is true we have%
\begin{align*}
& \left\vert \sqrt{n}\M_{M,n}-%
{\displaystyle\int\nolimits_{u\in D_{M}^{+}(\tau_{1})}}
\left(  \nabla_{x}(u)\beta_{n}^{X}(u)+\nabla_{y}(u)\beta_{n}^{Y}(u)\right)
du\right\vert \\
& \leqslant\sqrt{n}\M_{M,n}^{-}(\tau_{1})+\lambda(M,\tau_{1})\sqrt{n}%
\Theta\left(  \frac{4C_{\delta}}{\sqrt{n}}\right)
\end{align*}
where $\sqrt{n}\Theta\left(  4C_{\delta}\sqrt{n}\right)  \rightarrow0$ as
$n\rightarrow+\infty$. We also have $D_{M}^{-}(\tau_{1})
=I_M\setminus D_{M}^{+}(\tau_{1})\subset I_M$ and $\nabla
_{x},\nabla_{y}$ are bounded on $I_M$ thus%
\begin{align*}
& \left\vert
{\displaystyle\int\nolimits_{u\in D_{M}^{-}(\tau_{1})}}
\left(  \nabla_{x}(u)\beta_{n}^{X}(u)+\nabla_{y}(u)\beta_{n}^{Y}(u)\right)
du\right\vert \\
& \leqslant2m\frac{4C_{\delta}}{\sqrt{n}}\sup_{u\in I_M}\left\vert
\nabla_{x}(u)\right\vert +\left\vert \nabla_{y}(u)\right\vert.
\end{align*}
Hence under $A_{n}(M,C_{\delta})$ $\left\vert \sqrt{n}\M_{M,n}-%
\sqrt n \mathbb N_{M,n}\right\vert$ is bounded by 
$$ 4C_{\delta}d(m^{\prime},\tau_{1}^{\prime})+\lambda(M,\tau_{1})\sqrt{n}%
\Theta\left(  \frac{4C_{\delta}}{\sqrt{n}}\right)+2m\frac{4C_{\delta}}{\sqrt{n}}\sup_{u\in I_M}\left\vert
\nabla_{x}(u)\right\vert +\left\vert \nabla_{y}(u)\right\vert$$
Therefore, for any $\delta>0$, any $\varepsilon>0$ and any triplet
$M>m^{\prime}>m$ we can choose $\tau_{1}$ and $\tau_{1}^{\prime}>\tau_{1} $ so
small that $4C_{\delta}d(m^{\prime},\tau_{1}^{\prime})\leqslant\varepsilon/2$.
Then there exists $n(\varepsilon,\delta,M)$ such that for all $n>n(\varepsilon,\delta,M)$,%
\[
\mathbb{P}\left( \left\vert \sqrt{n}\M_{M,n}-%
\sqrt n \mathbb N_{M,n}\right\vert>\varepsilon\right)  \leqslant\mathbb{P}\left(  A_{n}(M,C_{\delta}%
)^{c}\right)  \leqslant\delta.
\]

\end{enumerate}
\end{proof}

\noindent\begin{bf}Step 5: Conclusion\end{bf}\\

\noindent Now recall that $\sqrt{n}\left(  W_{c}(\F_{n},\mathbb G_{n})-W_{c}(F,G)\right)
=\sqrt{n}D_{n}%
+\sqrt{n}S_{n}+\sqrt{n}T_{n}+\sqrt{n}U_{M,n}+\sqrt{n}\M_{M,n}$. By Steps 1, 2 $\sqrt{n}D_{n}
+\sqrt{n}S_{n}+\sqrt{n}T_{n}$ converges to zero in probability. Hence, we only need to prove  the weak convergence of  $\sqrt{n}U_{M,n}+\sqrt{n}\M_{M,n}$.  Let $\mathbb{X}_\infty$ be a centered Gaussian random variable with variance $\sigma^{2}(\Pi,c)$.
For any  $B$-bounded $r$-Lipschitz  function $\Phi$, we have

\begin{align*}
&\E\left[\left|\Phi\left( \sqrt{n}(U_{M,n}+\M_{M,n})\right)-\Phi\left(\mathbb{X}_\infty\right)\right|\right]\\&\leqslant \E\left[\left|\Phi\left( \sqrt{n}(U_{M,n}+\M_{M,n})\right)-\Phi\left(\sqrt{n}\M_{M,n}\right)\right|\right]
+\E\left[\left|\Phi\left( \sqrt{n}\M_{M,n}\right)-\Phi\left(\mathbb{X}_\infty\right)\right|\right]
\end{align*}
Dealing with the first right hand term we have

\begin{align*}
\E&\left[\left|\Phi\left( \sqrt{n}(U_{M,n}+\M_n)\right)-\Phi\left(\sqrt{n}\M_{M,n}\right)\right|\right]\\&=\E\left[\left|\Phi\left( \sqrt{n}(U_{M,n}+\M_{M,n})\right)-\Phi\left(\sqrt{n}\M_{M,n}\right)\right|\1_{| \sqrt{n}U_{M,n}|>\lambda}\right]\\
&+\E\left[\left|\Phi\left( \sqrt{n}(U_{M,n}+\M_{M,n})\right)-\Phi\left(\sqrt{n}\M_{M,n}\right)\right|\1_{| \sqrt{n}U_{M,n}|\leqslant\lambda}\right]\\
&\leqslant 2B\P\left(| \sqrt{n}U_{M,n}|>\lambda\right)r\lambda
\end{align*}
By lemma \ref{cvUMn} we can make $\displaystyle 2B\P\left(| \sqrt{n}U_{M,n}|>\lambda\right)r\lambda$ as small as we want by choosing $\lambda$ small enough and $M$ large enough.\\
We now consider the second right hand term
\begin{align*}
\E&\left[\left|\Phi\left( \sqrt{n}\M_{M,n}\right)-\Phi\left(\mathbb{X}_\infty\right)\right|\right]
\leqslant\\&\E\left[\left|\Phi\left( \sqrt{n}\M_{M,n}\right)-\Phi\left(\sqrt{n}\N_{M,n}\right)\right|\right]+\E\left[\left|\Phi\left( \sqrt{n}\N_{M,n}n\right)-\Phi\left(\mathbb{X}_\infty\right)\right|\right]
\end{align*}

By lemma \ref{MnNn} the term $\displaystyle\E\left[\left|\Phi\left( \sqrt{n}\M_{M,n}\right)-\Phi\left(\sqrt{n}\N_{M,n}\right)\right|\right]$ can be made as small as desired. As $\sqrt{n}\N_{M,n}$ is a Gaussian random variable with variance 

\begin{equation}
\sigma^{2}(M,\Pi,c)=%
{\displaystyle\int\nolimits_{F(-M)}^{F(M)}}
{\displaystyle\int\nolimits_{F(-M)}^{F(M)}}
\nabla(u)\Sigma(u,v)\nabla(v)dudv
\end{equation}
that converges to $\sigma^{2}(\Pi,c)$, the term $\displaystyle \E\left[\left|\Phi\left( \sqrt{n}\N_{M,n}\right)-\Phi\left(\mathbb{X}_\infty\right)\right|\right]$ is small enough for large enough $M$. This achieves the proof of Theorem \ref{mainth}.

\section{Appendix}\label{sec:app}

\subsection{Proof of auxiliary results}

\subsubsection{Proof of Lemma \ref{hongrois}}

Remind that $\Delta_{n}=\left[  \overline{u},1-k_{n}/n\right]  $ where
$k_{n}/\log\log n\rightarrow+\infty$ and $k_n/n\to 0$ comes from (\ref{kn}) and (\ref{knbis}).
Let us study $\left(  \beta_{n},h\right)  =(\beta_{n}^{X},h_{X})$ in Lemma \ref{hongrois}. Under
$(FG1)$ we have $f>0$\ on $\mathbb{R}$ thus the random variables
$U_{i}=F(X_{i})$ are independent, uniformely distributed on $\left[
0,1\right]  $ and such that $X_{(i)}=F^{-1}(U_{(i)})$. Let $\mathbb F_{U,n}$ and
$\mathbb F_{U,n}^{-1}$ denote the empirical cdf and quantile functions
associated to $U_{1},...,U_{n}$ so that $\mathbb F_{n}=\mathbb F_{U,n}\circ F$ and $\mathbb F_{n}%
^{-1}=F^{-1}\circ \mathbb F_{U,n}^{-1}$. Write $q_{n}(u)=\mathbb F_{U,n}^{-1}(u)-u$. By \cite{Csorg}
we have%
\begin{equation}
\underset{n\rightarrow\infty}{\lim\sup}\sup_{u\in\Delta_{n}}\frac{\sqrt
{n}q_{n}(u)}{\sqrt{\left(  1-u\right)  \log\log n}}\leqslant4\quad
a.s.\label{CSOREV78}%
\end{equation}
Since $(FG1)$ ensures that $h_{X}$ is $\mathcal{C}_{1}$ on $\Delta_{n}$ the
following expansion almost surely asymptotically holds,%
\begin{align}
&  \sup_{u\in\Delta_{n}}\left\vert \left(  F^{-1}(u+q_{n}(u))-F^{-1}%
(u)\right)  h_{X}(u)-q_{n}(u)\right\vert \nonumber\\
&  =\sup_{u\in\Delta_{n}}\left\vert \left(  \frac{q_{n}(u)}{h_{X}(u)}%
+\frac{q_{n}^{2}(u)}{2}\left(  \frac{1}{h_{X}(u)}\right)  _{u=u^{\ast}%
}^{\prime}\right)  h_{X}(u)-q_{n}(u)\right\vert \nonumber\\
&  \leqslant A_{n}B_{n}\nonumber
\end{align}
where $\left\vert u-u^{\ast}\right\vert \leqslant\left\vert q_{n}%
(u)\right\vert $ and, by (\ref{CSOREV78}),%
\[
A_{n}=\sup_{u\in\Delta_{n}}\frac{q_{n}^{2}(u)}{2\left(  1-u\right)  }%
\leqslant K\frac{\log\log n}{n}%
\]
whereas, by $(FG2)$,%
\begin{align*}
B_{n}  & =\sup_{u\in\Delta_{n}}\left(  1-u\right)  h_{X}(u)\left\vert \left(
\frac{1}{h_{X}(u)}\right)  _{u=u^{\ast}}^{\prime}\right\vert \\
& \leqslant\sup_{u\in\Delta_{n}}(1-u^{\ast})h_{X}(u^{\ast})\left\vert \left(
\frac{1}{h_{X}(u)}\right)  _{u=u^{\ast}}^{\prime}\right\vert \sup_{u\in
\Delta_{n}}\frac{1-u}{1-u^{\ast}}\frac{h_{X}(u)}{h_{X}(u^{\ast})}\\
& \leqslant K\sup_{u\in\Delta_{n}}\frac{1-u}{1-u^{\ast}}\sup_{u\in\Delta_{n}%
}\frac{h_{X}(u)}{h_{X}(u^{\ast})}.
\end{align*}
Now, (\ref{CSOREV78}) shows that the random sequence
\[
\sup_{u\in\Delta_{n}}\left\vert \frac{1-u^{\ast}}{1-u}-1\right\vert
\leqslant\sup_{u\in\Delta_{n}}\frac{1}{\sqrt{1-u}}\sup_{u\in\Delta_{n}}\left\vert
\frac{q_{n}(u)}{\sqrt{1-u}}\right\vert \leqslant5\sqrt{\frac{n}{k_{n}}}%
\sqrt{\frac{\log\log n}{n}}%
\]
almost surely tends to $0$. Moreover $(FG2)$ implies that%
\[
\left\vert \left(  \log h_{X}(u)\right)  ^{\prime}\right\vert \leqslant
K\left(  \log\frac{1}{1-u}\right)  ^{\prime}%
\]
so that $\left\vert \log h_{X}(u_{2})-\log h_{X}(u_{1})\right\vert \leqslant
K(\log(1-u_{1})-\log(1-u_{2}))$ for any $u_{1}<u_{2}$ in $\Delta_{n}$.
Therefore, the random sequence%
\[
\sup_{u\in\Delta_{n}}\frac{h_{X}(u)}{h_{X}(u^{\ast})}\leqslant\sup_{u\in
\Delta_{n}}\max\left((  \frac{1-u^{\ast}}{1-u}),(  \frac{1-u}{1-u^{\ast}})\right)^{K}%
\]
almost surely tends to $1$. We have shown that it almost surely ultimately
holds%
\begin{align*}
\sup_{u\in\Delta_{n}}\left\vert \frac{\beta_{n}^{X}(u)h_{X}(u)-\sqrt{n}%
q_{n}(u)}{\sqrt{(1-u)\log\log n}}\right\vert  & \leqslant A_{n}B_{n}%
\sqrt{\frac{n}{\log\log n}}\sup_{u\in\Delta_{n}}\frac{1}{\sqrt{1-u}}\\
& \leqslant10K\sqrt{\frac{\log\log n}{k_{n}}}%
\end{align*}
which proves Lemma \ref{hongrois}, by (\ref{CSOREV78}) again.
\subsubsection{Proof of Lemma \ref{lem:truc2}.}\label{sectruc2}

In view of (\ref{knbis}) and (\ref{In}) we eventually have $I_{n}\subset
\Delta_{n}$. Hence Lemma \ref{hongrois} and $(FG3)$ imply that, almost surely,
for all $n$ large
\begin{align*}
\sup_{u\in I_{n}}\frac{\left\vert \mathbb{F}_{n}^{-1}(u)-F^{-1}(u)\right\vert
}{F^{-1}(u)} &  \leqslant2K_{0}\sup_{u\in I_{n}}\frac{\sqrt{1-u}}%
{F^{-1}(u)h_{X}(u)}\sqrt{\frac{\log\log n}{n}}\\
&  =2K_{0}\sup_{u\in I_{n}}H_{X}(u)\sqrt{\frac{\log\log n}{n(1-u)}}\leqslant
K\sqrt{\frac{\log\log n}{k_{n}}}.
\end{align*}
The same bound holds for $\left\vert \mathbb{G}_{n}^{-1}(u)-G^{-1}%
(u)\right\vert /G^{-1}(u)$. By (\ref{knbis}) we then get%
\[
\lim_{n\rightarrow+\infty}\sup_{u\in I_{n}}\frac{\left\vert \varepsilon
_{n}^{Y}(u)\right\vert }{F^{-1}(u)}\leqslant\lim_{n\rightarrow+\infty}%
\sup_{u\in I_{n}}\frac{\left\vert \varepsilon_{n}^{Y}(u)\right\vert }%
{G^{-1}(u)}=\lim_{n\rightarrow+\infty}\sup_{u\in I_{n}}\frac{\left\vert
\varepsilon_{n}^{X}(u)\right\vert }{F^{-1}(u)}=0\quad a.s.
\]
so that $\sup_{u\in I_{n}}\left\vert \varepsilon_{n}(u)\right\vert /F^{-1}(u)$
almost surely vanishes. Under $(FG1)$ the law of large numbers for
$\mathbb{F}_{n}$ and $\mathbb{G}_{n}$ readily implies%
\[
\lim_{n\rightarrow+\infty}\mathbb{F}_{n}^{-1}\left(  1-\frac{h_{n}}{n}\right)
=\lim_{n\rightarrow+\infty}\mathbb{G}_{n}^{-1}\left(  1-\frac{h_{n}}%
{n}\right)  =+\infty\quad a.s.
\]
Therefore for any $q_{0}>0$, all $n$ large enough and all $u\in I_{n}$, it
holds%
\begin{equation}
\min\left(  \mathbb{F}_{n}^{-1}(u),F^{-1}(u),\mathbb{G}_{n}^{-1}%
(u),G^{-1}(u)\right)  >m,\quad\left\vert \varepsilon_{n}(u)\right\vert
<q_{0}F^{-1}(u)\label{ctorho}%
\end{equation}
which implies, by $(C2)$ and for $\tau(u)=F^{-1}(u)-G^{-1}(u)$,%
\[
c\left(  \mathbb{F}_{n}^{-1}(u),\mathbb{G}_{n}^{-1}(u)\right)  -c\left(
F^{-1}(u),G^{-1}(u)\right)  =\rho\left(  \left\vert \tau(u)+\varepsilon
_{n}(u)\right\vert \right)  -\rho\left(  \tau(u)\right)  .
\]
\textbf{Case 1.} Assume that $\gamma=0$ in $(C2)$. By Proposition \ref{convex}
$\rho^{\prime}$ is increasing and%
\[
\left\vert \rho\left(  \left\vert \tau(u)+\varepsilon_{n}(u)\right\vert
\right)  -\rho\left(  \tau(u)\right)  \right\vert \leqslant\rho^{\prime
}\left(  \tau(u)+\left\vert \varepsilon_{n}(u)\right\vert \right)  \left\vert
\varepsilon_{n}(u)\right\vert .
\]
Observe that if%
\[
\underset{u\rightarrow1}{\lim\inf}\frac{G^{-1}(u)}{F^{-1}(u)}=q_{1}>0
\]
then the result follows with $K_{2}=1$ since by taking $0<q_{0}<q_{1}%
\leqslant1$ in (\ref{ctorho}) we ultimately have, with probability one,%
\[
\rho^{\prime}\left(  \tau(u)+\left\vert \varepsilon_{n}(u)\right\vert \right)
=\rho^{\prime}\left(  F^{-1}(u)\left(  1-\frac{G^{-1}(u)}{F^{-1}(u)}%
+\frac{\left\vert \varepsilon_{n}(u)\right\vert }{F^{-1}(u)}\right)  \right)
\leqslant\rho^{\prime}\left(  F^{-1}(u)\right)  .
\]
If $q_{1}=0$, let us control $\rho^{\prime}\left(  \tau(u)+\left\vert
\varepsilon_{n}(u)\right\vert \right)  \leqslant\rho^{\prime}\left(
F^{-1}(u)\left(  1+\left\vert \varepsilon_{n}(u)\right\vert /F^{-1}(u)\right)
\right)  $. Remind (\ref{slow}) and the fact that $l$ is increasing whereas
$l^{\prime}$ is decreasing, by (\ref{L'}) and (\ref{L1}). For $y>x$,
$x\rightarrow+\infty$, $y\sim x$ we have $l(x)\leqslant l(y)\leqslant
l(2x)\sim l(x)$ and%
\[
\frac{\rho^{\prime}(y)}{\rho^{\prime}(x)}=\frac{l^{\prime}(y)}{l^{\prime}%
(x)}\frac{\rho(y)}{\rho(x)}\leqslant\frac{\rho(y)}{\rho(x)}=\exp\left(
l(y)-l(x)\right)  \leqslant\exp\left(  l^{\prime}(x)(y-x)\right)  .
\]
Therefore, by (\ref{L'}), (\ref{theta1}) and $(FG3)$, taking $\theta
_{1}^{\prime}\in(\theta_{1},\theta-1)$ yields%
\begin{align*}
1 &  \leqslant\frac{1}{\rho^{\prime}\circ F^{-1}(u)}\rho^{\prime}\left(
F^{-1}(u)\left(  1+\frac{\left\vert \varepsilon_{n}(u)\right\vert }{F^{-1}%
(u)}\right)  \right)  \\
&  \leqslant\exp\left(  l^{\prime}\circ F^{-1}(u)\left\vert \varepsilon
_{n}(u)\right\vert \right)  \\
&  =\exp\left(  \varepsilon_{1}\circ F^{-1}(u)l\circ F^{-1}(u)\frac{\left\vert
\varepsilon_{n}(u)\right\vert }{F^{-1}(u)}\right)  \\
&  \leqslant\exp\left(  \left(  l\circ F^{-1}\left(  1-\frac{k_{n}}{n}\right)
\right)  ^{\theta_{1}^{\prime}}K\sqrt{\frac{\log\log n}{k_{n}}}\right)
\end{align*}
provided $n$ is large enough and $u\in I_{n}$. Moreover (\ref{eq:lpsi})
implies%
\begin{equation}
l\circ F^{-1}\left(  1-\frac{k_{n}}{n}\right)  =l\circ\psi_{X}^{-1}\left(
\log\left(  \frac{n}{k_{n}}\right)  \right)  \leqslant l\circ\psi_{X}%
^{-1}\left(  \log n\right)  \leqslant\log n.\label{CFG1logn}%
\end{equation}
By choosing $\theta^{\prime}$ in (\ref{knbis}) such that $\theta
>\theta^{\prime}>1+\theta_{1}^{\prime}\geqslant\max(1,2\theta_{1}^{\prime})$
we get%
\[
\lim_{n\rightarrow+\infty}\sup_{u\in I_{n}}\frac{\rho^{\prime}\left(
\tau(u)+\left\vert \varepsilon_{n}(u)\right\vert \right)  }{\rho^{\prime}\circ
F^{-1}(u)} \leqslant1\quad a.s.
\]
which yields the result with $K_{2}=1$ again.\smallskip

\noindent\textbf{Case 2.} Assume that $\gamma>1$ in $(C2)$. Since $l^{\prime}$
is now increasing the above argument fails to guaranty that $\rho^{\prime
}(x)\sim\rho^{\prime}(y)$ as $y\sim x$ are sufficiently close. Instead we
check the sufficient condition in Proposition \ref{C3rho}. The function
$l(x)/x$ is increasing as it is regularly varying with index $\gamma-1>0$.
Recall also that $(CFG)$ yields (\ref{CFG1logn}) and that $H=H_{X}+H_{Y}$ is
bounded under $(FG3)$. As a consequence of $I_{n}\subset\Delta_{n}$ and Lemma
\ref{hongrois} we almost surely have, for all $n$ large,%
\begin{align}
\sup_{u\in I_{n}}\frac{l\circ\tau(u)}{\tau(u)}\left\vert \varepsilon
_{n}(u)\right\vert  &  \leqslant2K_{0}\sup_{u\in I_{n}}l\circ F^{-1}%
(u)H(u)\sqrt{\frac{\log\log n}{n(1-u)}}\nonumber\\
&  \leqslant2K_{0}l\circ F^{-1}\left(  1-\frac{k_{n}}{n}\right)  \sqrt
{\frac{\log\log n}{k_{n}}}\sup_{u\in I_{n}}H(u)\nonumber\\
&  \leqslant K\frac{\log n}{\sqrt{k_{n}}}\sqrt{\log\log n}\sup_{u\in I_{n}%
}H(u).\label{DL1}%
\end{align}
Since 
$\theta>2$ in $(CFG)$ choosing $\theta^{\prime}\in\left(  2,\theta\right)  $ in
(\ref{knbis}) makes the upper bound in (\ref{DL1}) vanish. Therefore, under
$(CFG)$ the requirements of Proposition \ref{C3rho} are almost surely
ultimately fulfiled with
\[
x_{0}=\tau_{0},\quad x=\tau(u),\quad\left\vert \varepsilon\right\vert
=\left\vert \varepsilon_{n}(u)\right\vert \leqslant\frac{\delta_{0}}%
{l^{\prime}(x)}=\frac{\delta_{0}}{l^{\prime}\circ\tau(u)}=\frac{\delta_{0}%
\tau(u)}{\gamma l\circ\tau(u)},\quad u\in I_{n},
\]
which entails that, for all $n$ large enough and $K_{2}=k_{0}$,%
\begin{equation}
\left\vert \rho\left(  \left\vert \tau(u)+\varepsilon_{n}(u)\right\vert
\right)  -\rho\left(  \tau(u)\right)  \right\vert \leqslant k_{0}\rho^{\prime
}\circ\tau(u)\left\vert \varepsilon_{n}(u)\right\vert \leqslant K_{2}%
\rho^{\prime}\circ F^{-1}(u)\left\vert \varepsilon_{n}(u)\right\vert
.\label{case23}%
\end{equation}
\textbf{Case 3.}\ Assume that $0<\gamma\leqslant1$ in $(C2)$. Since $l(x)/x$
is either decreasing or, if $\gamma=1$, not even monotone, $l\circ\tau
(u)/\tau(u)$ cannot be compared to the worse case $\tau(u)\sim F^{-1}(u)$
directly. However, by Proposition \ref{C3rho}, if $u\in I_{n}$ is such that
$\left\vert \varepsilon_{n}(u)\right\vert \leqslant\delta_{0}/l^{\prime}%
(\tau(u))$ then (\ref{case23}) holds. Consider%
\[
I_{n}^{-}=\left\{  u\in I_{n}:\left\vert \varepsilon_{n}(u)\right\vert
>\frac{\delta_{0}}{l^{\prime}\circ\tau(u)}\right\}  .
\]
Since $l^{\prime}(x)\sim\gamma l(x)/x$ and $\rho(x)\sim x\rho^{\prime
}(x)/\gamma l(x)$ as $x\rightarrow+\infty$, for any $0<x_{0}<\tau_{0}$ we can
find $\xi_{0}>1/\gamma$ such that%
\begin{equation}
\rho(x)\leqslant\xi_{0}\rho^{\prime}(x)\frac{x}{l(x)},\quad x\geqslant
x_{0}.\label{rhocas3}%
\end{equation}
Let $\xi_{1}>\gamma/\delta_{0}$ and assume $n$ so large that $l\circ
\tau(u)>1/\xi_{1}$ and $\tau(u)\geqslant\tau_{0}$ for $u\in I_{n}$. Any $u\in
I_{n}^{-}$ then satisfies%
\begin{align*}
\tau_{0}  & \leqslant\max\left(  \tau(u),\left\vert \varepsilon_{n}%
(u)\right\vert \right)  \\
& \leqslant\max\left(  \delta_{0}\xi_{1}\frac{l\circ\tau(u)}{l^{\prime}%
\circ\tau(u)},\left\vert \varepsilon_{n}(u)\right\vert \right)  \\
& \leqslant\frac{x_{n}(u)}{2}:=\xi_{1}l\circ\tau(u)\left\vert \varepsilon
_{n}(u)\right\vert .
\end{align*}
By (\ref{rhocas3}) and the fact that $l(x)$ is increasing it follows that%
\begin{align*}
\left\vert \rho\left(  \left\vert \tau(u)+\varepsilon_{n}(u)\right\vert
\right)  -\rho\left(  \tau(u)\right)  \right\vert  &  \leqslant\rho\left(
\tau(u)+\left\vert \varepsilon_{n}(u)\right\vert \right)  \\
&  \leqslant\xi_{0}\rho^{\prime}(x_{n}(u))\frac{x_{n}(u)}{l\circ\tau(u)}\\
&  =2\xi_{0}\xi_{1}\rho^{\prime}(x_{n}(u))\left\vert \varepsilon
_{n}(u)\right\vert .
\end{align*}
Using (\ref{CFG1logn}) as for (\ref{DL1}) we almost surely eventually have%
\[
\frac{1}{2\xi_{1}}\sup_{u\in I_{n}^{-}}\frac{x_{n}(u)}{F^{-1}(u)}=\sup_{u\in
I_{n}^{-}}l\circ\tau(u)\frac{\left\vert \varepsilon_{n}(u)\right\vert }%
{F^{-1}(u)}\leqslant K\frac{\log n}{\sqrt{k_{n}}}\sqrt{\log\log n}\sup_{u\in
I_{n}^{-}}H(u)
\]
and the upper bound tends to $0$ provided that $2<\theta^{\prime}<\theta$ from
(\ref{knbis}). As a
conclusion, $x_{n}(u)\leqslant F^{-1}(u)$ on $I_{n}^{-}$ even if $\left\vert
\varepsilon_{n}(u)\right\vert $ is large and it asymptotically holds, for
$K_{2}=\max(k_{0},2\xi_{0}\xi_{1})$,%
\[
\left\vert \rho\left(  \left\vert \tau(u)+\varepsilon_{n}(u)\right\vert
\right)  -\rho\left(  \tau(u)\right)  \right\vert \leqslant K_{2}\rho^{\prime
}(F^{-1}(u))\left\vert \varepsilon_{n}(u)\right\vert ,\quad u\in I_{n}.
\]

\subsubsection{Strong approximation of the joint quantile processes}

In this section $(FG1)$ and $(FG2)$ are crucially required to justify the key
approximation used at steps 4 and of the main proof. Let $k_{n}$ be defined as
in (\ref{kn}), thus $k_{n}/n\rightarrow0$, $k_{n}/\log\log n\rightarrow+\infty$.
Consider $\mathcal{I}_{n}=\left(  k_{n}/n,1-k_{n}/n\right)  $ which contains
both $I_{M,n}$ from (\ref{Imn}) and $\Delta_{n}$ from (\ref{In}).\ As in (\ref{betan}) write
$\beta_{n}^{X}=\sqrt{n}(\mathbb{F}_{n}^{-1}-F^{-1})$ and $\beta_{n}^{Y}%
=\sqrt{n}(\mathbb{G}_{n}^{-1}-G^{-1})$ the quantile processes associated to
each sample. Our goal is to derive a coupling of
\[
\left\{  (\beta_{n}^{X}(u),\beta_{n}^{Y}(u)):u\in\mathcal{I}_{n}\right\}
\quad\text{and}\quad\left\{  \left(  \frac{B_{n}^{X}(u)}{h_{X}(u)},\frac
{B_{n}^{Y}(u)}{h_{Y}(u)}\right)  :u\in\mathcal{I}_{n}\right\}
\]
where $(B_{n}^{X},B_{n}^{Y})$ are two marginal standard Brownian Bridges%
\begin{align}\label{ponts}
B_{n}^{X}(u)  & =\mathbb{B}_{n}\left(  \mathcal{H}_{F^{-1}(u)}\right)
,\quad\mathcal{H}_{x_{0}}=\left\{  (x,y):x\leqslant x_{0}\right\}  ,\\
B_{n}^{Y}(u)  & =\mathbb{B}_{n}\left(  \mathcal{H}^{G^{-1}(u)}\right)
,\quad\mathcal{H}^{y_{0}}=\left\{  (x,y):y\leqslant y_{0}\right\}  ,
\end{align}
indexed by $u\in\left[  0,1\right]  $ and driven by a sequence $\mathbb{B}%
_{n}$ of $\Pi$-Brownian Bridge indexed by the collection $\mathcal{C}$ of half
planes $\mathcal{H}_{x_{0}}$ or $\mathcal{H}^{y_{0}}$. In other words,
$\mathbb{B}_{n}$ is a zero mean Gaussian process indexed by $\mathcal{C}$
having covariance%
\[
cov(\mathbb{B}_{n}(A),\mathbb{B}_{n}(B))=\Pi(A\cap B)-\Pi(A)\Pi(B)
\]
for $A,B\in\mathcal{C}$, and $B_{n}^{X}$ are centered Gaussian processes with
covariance%
\begin{align*}
cov(B_{n}^{X}(u),B_{n}^{X}(v))  & =\Pi(\mathcal{H}_{F^{-1}(u)}\cap
\mathcal{H}_{F^{-1}(v)})-uv=\min(u,v)-uv\\
cov(B_{n}^{Y}(u),B_{n}^{Y}(v))  & =\Pi(\mathcal{H}^{G^{-1}(u)}\cap
\mathcal{H}^{G^{-1}(v)})-uv=\min(u,v)-uv\\
cov(B_{n}^{X}(u),B_{n}^{Y}(v))  & =\Pi(\mathcal{H}_{F^{-1}(u)}\cap
\mathcal{H}^{G^{-1}(v)})-uv=L(u,v)-uv
\end{align*}
for $u,v\in\left[  0,1\right]  $, where the usual copula function
$L(u,v)=H(F^{-1}(u),G^{-1}(v))$ measures the distortion between $\Pi$ and
$P\otimes Q$ on all quadrants, half spaces and then rectangles.

The coupling is achieved at Theorem \ref{approx} simply by combining the strong approximation of the empirical process (see \cite{Ber})
$$\Lambda_{n}(A)=\sqrt{n}(\Pi_{n}(A)-\Pi(A)),\  A\in {\cal C},\  \Pi_{n}
=\frac{1}{n}\sum_{i\leqslant n}\delta_{(X_{i},Y_{i})}$$  with the
usual quantile transform and classical results for real quantiles. This result
has an interest by itself as it is valid whatever the joint law $\Pi$
satisfying the marginal conditions $(FG1)$ and $(FG2)$. 
\begin{remark}
Theorem \ref{approx} remains valid for the $d$ marginal quantile processes of a
law $\Pi$ in $\mathbb{R}^{d}$ provided each marginal laws obeys $(FG1)$ and
$(FG2)$, with obviously no change in the proof for $d=2$. 
\end{remark}

\begin{theorem}
\label{approx} Assume  that $F,G$ satisfy  $(FG1)$ and $(FG2)$.
One can built on the same probability space the sequence $\left\{
(X_{n},Y_{n})\right\}  $ and a sequence of versions of $\left\{  (B_{n}%
^{X}(u),B_{n}^{Y}(u)):u\in\mathcal{I}_{n}\right\}  $ such that%
\[
\beta_{n}^{X}(u)=\frac{B_{n}^{X}(u)+Z_{n}^{X}(u)}{h_{X}(u)},\quad\beta_{n}%
^{Y}(u)=\frac{B_{n}^{Y}(u)+Z_{n}^{Y}(u)}{h_{Y}(u)}%
\]
satisfies, for some $\xi>0$,%
\[
\lim_{n\rightarrow+\infty}n^{\xi}\sup_{u\in\mathcal{I}_{n}}\left\vert Z_{n}%
^{X}(u)\right\vert =\lim_{n\rightarrow+\infty}n^{\xi}\sup_{u\in\mathcal{I}_{n}%
}\left\vert Z_{n}^{Y}(u)\right\vert =0\quad a.s.
\]
Moreover we can take%
\[
(B_{n}^{X}(u),B_{n}^{Y}(u))=\frac{1}{\sqrt{n}}%
{\displaystyle\sum\limits_{k=1}^{n}}
(G_{k}^{X}(u),G_{k}^{Y}(u))
\]
where $\left\{  (G_{k}^{X}(u),G_{k}^{Y}(u)):u\in\left(  0,1\right)  \right\}
$ is a sequence of independent versions of Brownian Bridges $(G^{X},G^{Y})$
such that $cov(G^{X}(u),G^{Y}(v))=L(u,v)-uv$.
\end{theorem}

\begin{proof}
Define the two marginal empirical processes to be, for $x\in\mathbb{R}$,%
\begin{align*}
\alpha_{n}^{X}(x)  & =\sqrt{n}(\mathbb{F}_{n}(x)-F(x))=\Lambda_{n}%
(\mathcal{H}_{x}),\\
\alpha_{n}^{Y}(x)  & =\sqrt{n}(\mathbb{G}_{n}(x)-G(x))=\Lambda_{n}%
(\mathcal{H}^{x}).
\end{align*}
Under $(FG1)$ the random variables $U_{i}=F(X_{i})$ and $V_{i}=G(Y_{i})$ are
uniform on $\left(  0,1\right)  $. Write $\alpha_{n}^{X,U}$ and $\alpha
_{n}^{Y,V}$ the uniform empirical process associated to $U_{1},...,U_{n}$ and
$V_{1},...,V_{n}$ respectively. Also write $\mathbb{F}_{X,U,n}$ and
$\mathbb{F}_{X,U,n}^{-1}$ the empirical c.d.f. and quantile functions
then $\beta_{n}^{X,U}(u)=\sqrt{n}(\mathbb{F}_{X,U,n}^{-1}(u)-u)$. Likewise
write $\mathbb{F}_{Y,V,n}$, $\mathbb{F}_{Y,V,n}^{-1}$ and $\beta_{n}^{Y,V}$.
Clearly $\alpha_{n}^{X,U}$ and $\alpha_{n}^{Y,V}$ are not independent, neither
are $\beta_{n}^{X,U}$ and $\beta_{n}^{Y,V}$. What is next obtained for $X$ is
also valid for $Y$.

Under $(FG1)$ and $(FG2)$ the arguments given in Section 5.1.1 yield that%
\begin{equation}
\underset{n\rightarrow+\infty}{\lim}\ \frac{\sqrt{n}}{\log\log n}\sup
_{u\in\mathcal{I}_{n}}\left\vert h_{X}(u)\beta_{n}^{X}(u)-\beta_{n}%
^{X,U}(u)\right\vert =0\quad a.s.\label{betaUXn}%
\end{equation}
since $\beta_{n}^{X,U}=\sqrt{n}q_{n}$ and the supremum is showed to be less
than $\sqrt{n}A_{n}B_{n}$ with the almost sure bounds such that $A_{n}%
<K(\log\log n)/n$ and $B_{n}\rightarrow0$ as $n\rightarrow+\infty$. By \cite{Baha} 
and \cite{Kief} we also have%
\begin{equation}
\underset{n\rightarrow+\infty}{\lim\sup}\ \frac{n^{1/4}}{\sqrt{\log n}%
(\log\log n)^{1/4}}\sup_{u\in\mathcal{I}_{n}}\left\vert \beta_{n}%
^{X,U}(u)+\alpha_{n}^{X,U}(u)\right\vert \leqslant\frac{1}{2^{1/4}}\quad
a.s.\label{bahadur}%
\end{equation}
thus for any $\xi<1/4$ it holds%
\[
\underset{n\rightarrow+\infty}{\lim}\ n^{\xi}\sup_{u\in\mathcal{I}_{n}%
}\left\vert h_{X}(u)\beta_{n}^{X}(u)+\alpha_{n}^{X,U}(u)\right\vert =0\quad
a.s.
\]
It is important here that (\ref{betaUXn}) and (\ref{bahadur}) holds true for
$\beta_{n}^{X,U}$ and $\beta_{n}^{Y,U}$ simultaneously with probability one
whatever the underlying probability space. Hence, recalling that $\alpha
_{n}^{X,U}=\alpha_{n}^{X}\circ F^{-1}$, $\Pi_{n}(\mathcal{H}_{F^{-1}%
(u)})=\mathbb{F}_{n}(F^{-1}(u))$ and $\Pi(\mathcal{H}_{F^{-1}(u)})=u$ it
follows that%
\begin{align*}
\underset{n\rightarrow+\infty}{\lim}\ n^{\xi}\sup_{u\in\mathcal{I}_{n}%
}\left\vert h_{X}(u)\beta_{n}^{X}(u)+\Lambda_{n}(\mathcal{H}_{F^{-1}%
(u)})\right\vert  & =0\quad a.s.\\
\underset{n\rightarrow+\infty}{\lim}\ n^{\xi}\sup_{u\in\mathcal{I}_{n}%
}\left\vert h_{Y}(u)\beta_{n}^{Y}(u)+\Lambda_{n}(\mathcal{H}^{G^{-1}%
(u)})\right\vert  & =0\quad a.s.
\end{align*}
on any probability space. It remains to approximate $\Lambda_{n}$ uniformly on
$\mathcal{C}$. The collection of sets $\mathcal{C}$ is a VC-class of order $3$
thus satisfies the uniform entropy condition $(VC)$ used in \cite{Ber}
with $v_{0}=2(3-1)=4$. By their Proposition 1 taken with $\theta=2$
there exists a probability space on which the sequence $\left\{  (X_{n}%
,Y_{n})\right\}  $ can be built together with a sequence $\mathbb{B}_{n}$ of
$\Pi$-Brownian Bridges indexed by $\mathcal{C}$ such that%
\[
\mathbb{P}\left(  \sup_{A\in\mathcal{C}}\left\vert \Lambda_{n}(A)+\mathbb{B}%
_{n}(A)\right\vert \geqslant\frac{K}{n^{\beta_{2}}}\right)  \leqslant\frac
{1}{n^{2}}%
\]
where we take $\beta_{2}>1/22$ to avoid the $\log n$ factor. Note that since
$\mathbb{B}_{n}$ and $-\mathbb{B}_{n}$ have the same law, we choose to
approximate with $-\mathbb{B}_{n}$. Consider in particular $\mathcal{H}%
_{n}^{X}=\left\{  \mathcal{H}_{F^{-1}(u)}:u\in\mathcal{I}_{n}\right\}
\subset\mathcal{C}$ and define $B_{n}^{X}(u)=\mathbb{B}_{n}\left(
\mathcal{H}_{F^{-1}(u)}\right)  $. On the previous probability space it holds%
\begin{align*}
& \underset{n\rightarrow+\infty}{\lim\sup}\ n^{\beta_{2}}\sup_{u\in
\mathcal{I}_{n}}\left\vert \alpha_{n}^{X}\circ F^{-1}(u)+B_{n}^{X}%
(u)\right\vert \\
& =\underset{n\rightarrow+\infty}{\lim\sup}\ n^{\beta_{2}}\sup_{A\in
\mathcal{H}_{n}^{X}}\left\vert \Lambda_{n}(A)+\mathbb{B}_{n}(A)\right\vert
\leqslant K\quad a.s.
\end{align*}
the above comparision between $h_{X}(u)\beta_{n}^{X}(u)$ and $\alpha_{n}%
^{X}\circ F^{-1}(u)$ gives in turn, for $\xi<\max(1/4,\beta_{2})=\beta_{2}$
and $Z_{n}^{X}(u)=h_{X}(u)\beta_{n}^{X}(u)-B_{n}^{X}(u)$,%
\[
\underset{n\rightarrow+\infty}{\lim\sup}\ n^{\xi}\sup_{u\in\mathcal{I}_{n}%
}\left\vert Z_{n}^{X}(u)\right\vert =0\quad a.s.
\]
In the same way we simultaneously obtain, for $Z_{n}^{Y}(u)=h_{Y}(u)\beta
_{n}^{Y}(u)-B_{n}^{Y}(u)$,%
\[
\underset{n\rightarrow+\infty}{\lim\sup}\ n^{\xi}\sup_{u\in\mathcal{I}_{n}%
}\left\vert Z_{n}^{Y}(u)\right\vert =0\quad a.s.
\]
The processes $B_{n}^{X}$ and $B_{n}^{Y}$ are joint through the leading
process $\mathbb{B}_{n}$, whence the covariance $cov(B_{n}^{X}(u),B_{n}%
^{Y}(v))=L(u,v)-uv$. The second statement to be proved follows by applying
Theorem 1 of \cite{Ber} in place of Proposition 1. If $\beta
_{2}>0$ is chosen small enough the approximating process can be built in the
form $\mathbb{B}_{n}=\sum\nolimits_{k=1}^{n}\mathbb{B}_{k}^{\ast}/\sqrt{n}$
where $\left\{  \mathbb{B}_{k}^{\ast}:k\geqslant1\right\}  $ is a sequence of
independent $\Pi$-Brownian Bridges. Since $\mathbb{B}_{n}$ is again a $\Pi
$-Brownian Bridge, $G_{k}^{X}(u)=\mathbb{B}_{k}^{\ast}(\mathcal{H}_{F^{-1}%
(u)})$ and $G_{k}^{Y}(u)=\mathbb{B}_{k}^{\ast}(\mathcal{H}^{G^{-1}(u)})$ are
standard Brownian Bridges with the desired correlation structure.
\end{proof}

\subsection{Complements on assumptions}

\subsubsection{Regular an smooth slow variation}\label{regvar}

\noindent In this section, we present the regular and slow variation properties needed for assumption (C2). For more details we refer to \cite{Miko,Sene}. For $k\in
\mathbb{N}_{\ast}$ write $\mathcal{C}_{k}$ the set of functions that are $k$
times continuously differentiable on $\mathbb{R}$, and $\mathcal{C}_{0}$ the
set of continuous functions. Let $\mathcal{M}%
_{k}\left(  m,+\infty\right)  $ be the subset of functions $\varphi
\in\mathcal{C}_{k}$ such that $\varphi^{(k)}$ is monotone on $\left(
m,+\infty\right)  $, and hence $\varphi$, $\varphi^{\prime}$, $\varphi
^{\prime\prime}$,...,$\varphi^{(k)}$ are also monotone on $\left(
m,+\infty\right)  $ by changing $m$. Let $\mathcal{M}_{0}\left(
m,+\infty\right)  $ denote the set of continuous functions monotone on
$\left(  m,+\infty\right)  $. Write $RV(\gamma)$ the set of regularly varying
functions at $+\infty$ with index $\gamma\in\mathbb{R}$. They are of the form
$x^{\gamma}L(x)$ with $L\in RV(0)$, which means that given any $\lambda>0$,%
\begin{equation}
\lim_{x\rightarrow+\infty}\frac{L(\lambda x)}{L(x)}=1.\label{slow}%
\end{equation}
If $L\in RV(0)$ is monotone on $\left(  m,+\infty\right)  $ then $L$ is
equivalent at $+\infty$ to a function in $\mathcal{C}_{\infty}\left(
m,+\infty\right)  \cap RV(0)$. Therefore, at the first order, it is not a
restriction to assume that functions of $RV(\gamma)$ are in $\mathcal{M}%
_{k}\left(  m,+\infty\right)  $ as well. Problems however arise with respect
to differentiation. In particular, two apparently close slowly varying
functions may have very different local variations. First consider the smooth
regular variation. Let introduce%
\[
RV_{k}(\gamma,m)=RV(\gamma)\cap\mathcal{M}_{k}\left(  m,+\infty\right)
,\quad\gamma\neq0.
\]
The following statements are taken as $x\rightarrow+\infty$. Assuming that
$k\geqslant1$ and $\gamma\neq0$, if $\varphi\in RV_{k}(\gamma,m)$ then
$\varphi^{\prime}$ is monotone, so that it holds, by the monotone density
theorem,%
\begin{equation}
\varphi^{\prime}(x)\sim\frac{\gamma\varphi(x)}{x}.\label{phiprime}%
\end{equation}
This implies that $\varphi^{\prime}\in RV_{k-1}(\gamma-1,m)$ and, whenever
$k\geqslant2$ and $\gamma\neq1$, $\varphi^{\prime\prime}$ in turns satisfies
$\varphi^{\prime\prime}\in RV_{k-2}(\gamma-2,m)$ and%
\begin{equation}
\varphi^{\prime\prime}(x)\sim\frac{(\gamma-1)\varphi^{\prime}(x)}{x}\sim
\frac{\gamma(\gamma-1)\varphi(x)}{x^{2}}.\label{phiseconde}%
\end{equation}
 For $L\in RV(0)$ it holds, by Karamata's theorem,%
\[
\frac{\int_{m}^{x}L^{\prime}(t)\left(  \frac{L(t)}{tL^{\prime}(t)}\right)
dt}{\int_{m}^{x}L^{\prime}(t)dt}=\frac{1}{L(x)}\int_{m}^{x}\frac{L(t)}%
{t}dt\rightarrow+\infty.
\]
Hence the function $L(t)/tL^{\prime}(t)$ is unbounded and, if $L\in
\mathcal{C}_{1}\left(  m,+\infty\right)  $, continuous on $\left(
m,+\infty\right)  $. It is not very restrictive to exclude functions
$L(t)/tL^{\prime}(t)$ that are asymptotically oscillating and not going to
infinity. We thus assume (\ref{L'}).\\%
For instance, if $L(x)=\varphi(\log x)$ where $\varphi\in RV_{2}(\gamma,m)$
and $\gamma>0$ then $\varepsilon_{1}(x)\sim\gamma/\log x$. Likewise, if
$L(x)=\varphi(L_{1}(x))$ where $\varphi\in RV_{2}(\gamma,m)$ and $\gamma>0$
then we get $\varepsilon_{1}(x)\sim\gamma xL_{1}^{\prime}(x)/L_{1}(x)$. Also
remind the well known representation, for $x\in\left(  m,+\infty\right)  $,%
\[
L(x)=d_{0}(x)\exp\left(  \int_{m}^{x}\frac{\varepsilon_{0}(t)}{t}dt\right)
,\quad d_{0}(x)\rightarrow d_{0}>0,\quad\varepsilon_{0}(x)\rightarrow0.
\]
If $d_{0}(x)$ is constant then $d_{0}=L(m)$ and $\varepsilon_{0}%
(x)=\varepsilon_{1}(x)$ from (\ref{L'}). More generally, (\ref{L'}) is
equivalent to $xd_{0}^{\prime}(x)\rightarrow0$ and we have $\varepsilon
_{1}(x)=\varepsilon_{0}(x)+xd_{0}^{\prime}(x)$.

\subsubsection{A sufficient condition for $(FG)$}\label{SCFG}

In this section we provide a sufficient condition to%
\[
(FG2)\;\sup_{x>m}\left(  1-F(x)\right)  \frac{\left\vert f^{\prime
}(x)\right\vert }{f^{2}(x)}<+\infty,\quad(FG3)\;\sup_{u>F(m)}H_{X}%
(u)<+\infty,
\]
based on standard regular variation or smooth slow variation. Starting from%
\begin{align*}
F(x)  &  =1-\exp(-\psi_{X}(x)),\quad F^{-1}(u)=\psi_{X}^{-1}(\log(1/(1-u))),\\
f(x)  &  =\psi_{X}^{\prime}(x)\exp(-\psi_{X}(x)),\quad h_{X}(u)=(1-u)\psi
_{X}^{\prime}\circ\psi_{X}^{-1}(\log(1/(1-u))),
\end{align*}
we have%
\[
H_{X}(u)=\frac{1-u}{F^{-1}(u)h_{X}(u)}=\left(  \log\psi_{X}^{-1}\right)
^{\prime}(\log(1/(1-u)))
\]
thus $(FG3)$ holds whenever $\left(  \log\psi_{X}^{-1}\right)  ^{\prime}(x)$
is bounded, or $1/x\psi_{X}^{\prime}(x)$ is bounded. Conversely, $(FG3)$
implies that $F^{-1}(u)=O(1/\left(  1-u\right)  ^{K})$ for $K$ bounding
$H_{X}$ since $\left(  \log F^{-1}(u)\right)  ^{\prime}=H_{X}(u)/(1-u)$. In
the same vein, $(FG2)$ is equivalent to%
\begin{equation}
\sup_{m<x<+\infty}\left\vert \left(  \frac{1}{\psi_{X}^{\prime}(x)}\right)
^{\prime}\right\vert <+\infty\label{FG2bis}%
\end{equation}
since $f^{\prime}(x)=\left(  -\psi_{X}^{\prime\prime}(x)-\psi_{X}^{\prime
2}(x)\right)  \exp(-\psi_{X}(x))$ and%
\[
\left(  1-F(x)\right)  \frac{\left\vert f^{\prime}(x)\right\vert }{f^{2}%
(x)}=\left\vert \frac{\psi_{X}^{\prime\prime}(x)+\psi_{X}^{\prime2}(x)}%
{\psi_{X}^{\prime2}(x)}\right\vert =\left\vert \frac{\psi_{X}^{\prime\prime
}(x)}{\psi_{X}^{\prime2}(x)}+1\right\vert =\left\vert \left(  \frac{1}%
{\psi_{X}^{\prime}(x)}\right)  ^{\prime}-1\right\vert .
\]

\begin{proposition}
\label{CSFG}If $\psi_{X}\in RV_{2}^{+}(0,m)$ then $F$ satisfies $(FG)$. If
$\psi_{X}\in RV_{2}(\gamma_{1},m)$ for some $\gamma_{1}>\gamma_{0}>0$ and, if
$\gamma_{1}=1$ assuming also that $\psi_{X}(x)=xL(x)$ with $L^{\prime}\in
RV_{1}(-1,m)$ and (\ref{L'}), then $F$ satisfies $(FG)$ and $(FG3)$ can be
replaced by%
\begin{equation}
H_{X}(u)\leqslant\frac{1}{\gamma_{0}\log(1/(1-u))},\quad u>F(m). \label{HXreg}%
\end{equation}

\end{proposition}

\begin{proof}
Clearly $\psi_{X}$ is $\mathcal{C}_{2}$, increases to infinity,
$F(x)=1-e^{-\psi_{X}(x)}$ has unbounded right tail and $f(x)=\psi_{X}^{\prime
}(x)\exp(-\psi_{X}(x))$ is $\mathcal{C}_{1}$ which yields $(FG1)$. Next we
check $(FG2)$ and $(FG3)$ in the two cases.\newline\textbf{Case 1.}\ Assume
that $\gamma_{1}>0$. If $\gamma_{1}\neq1$ then (\ref{phiprime}) and
(\ref{phiseconde}) give, as $x\rightarrow+\infty$,%
\[
\left(  \frac{1}{\psi_{X}^{\prime}(x)}\right)  ^{\prime}=\frac{\psi
_{X}^{\prime\prime}(x)}{\psi_{X}^{\prime2}(x)}\sim\frac{\gamma_{1}-1}%
{\gamma_{1}\psi_{X}(x)}\rightarrow0.
\]
If $\gamma_{1}=1$ then $\psi_{X}(x)=xL(x)$ with $L\in RV_{2}(0,m)$,
$L^{\prime}\in RV_{1}(-1,m)$ and (\ref{L'}) thus $L^{\prime}(x)\sim
-xL^{\prime\prime}(x)$ and $xL^{\prime}(x)/L(x)\rightarrow0$ which entails%
\[
\left\vert \left(  \frac{1}{\psi_{X}^{\prime}(x)}\right)  ^{\prime}\right\vert
=\left\vert \frac{\psi_{X}^{\prime\prime}(x)}{\psi_{X}^{\prime2}%
(x)}\right\vert =\frac{\left\vert 2L^{\prime}(x)+xL^{\prime\prime
}(x)\right\vert }{(L(x)+xL^{\prime}(x))^{2}}\leqslant K\frac{\left\vert
L^{\prime}(x)\right\vert }{L^{2}(x)}\rightarrow0.
\]
Whence (\ref{FG2bis}) and $(FG2)$.\\
Whatever $\gamma_{1}>0$, $\psi_{X}$ is
continuous and strictly increasing with inverse $\psi_{X}^{-1}\in
RV_{2}(1/\gamma_{1},\psi_{X}^{-1}(m))$.\ By using again (\ref{phiprime}) we
obtain%
\[
\left(  \log\psi_{X}^{-1}(x)\right)  ^{\prime}=\frac{1}{\psi_{X}^{-1}%
(x)\psi_{X}^{\prime}\circ\psi_{X}^{-1}(x)}\sim\frac{1}{\gamma_{1}\psi_{X}%
\circ\psi_{X}^{-1}(x)}=\frac{1}{\gamma_{1}x}%
\]
as $x\rightarrow+\infty$, which is bounded. This implies $(FG3)$ and more
accurately (\ref{HXreg}) since for $\gamma_{1}>\gamma_{0}>0$ and $m$
sufficiently large,%
\[
H_{X}(u)=\left(  \log\psi_{X}^{-1}\right)  ^{\prime}(\log(1/(1-u)))\leqslant
\frac{1}{\gamma_{0}\log(1/(1-u))},\quad u>F(m).
\]
\textbf{Case 2.}\ If $\gamma_{1}=0$ then $x\psi_{X}^{\prime}(x)=\varepsilon
_{1}(x)\psi_{X}(x)\geqslant l_{1}\geqslant1$ with $\varepsilon_{1}%
(x)\rightarrow0$ as $x\rightarrow+\infty$, by (\ref{L'}) and (\ref{L1}). Now
$l_{1}/x\leqslant\psi_{X}^{\prime}(x)\leqslant\psi_{X}(x)/x$ implies that
$\psi_{X}^{\prime}\in RV_{1}(-1,m)$ and (\ref{phiseconde}) yields $\psi
_{X}^{\prime\prime}(x)\sim-\psi_{X}^{\prime}(x)/x=-\varepsilon_{1}(x)\psi
_{X}(x)/x^{2}$. It ensues $\left(  1/\psi_{X}^{\prime}(x)\right)  ^{\prime
}=\psi_{X}^{\prime\prime}(x)/\psi_{X}^{\prime2}(x)\sim-1/\varepsilon
_{1}(x)\psi_{X}(x).$ Therefore the upper bound in $(FG2)$ is, for $x\in\left(
m,+\infty\right)  $,%
\[
\left\vert \left(  \frac{1}{\psi_{X}^{\prime}(x)}\right)  ^{\prime
}-1\right\vert =1+\frac{1}{\varepsilon_{1}(x)\psi_{X}(x)}\leqslant\frac
{l_{1}+1}{l_{1}}%
\]
and the upper bound in $(FG3)$ is%
\[
\left(  \log\psi_{X}^{-1}(x)\right)  ^{\prime}=\frac{1}{\psi_{X}^{-1}%
(x)\psi_{X}^{\prime}\circ\psi_{X}^{-1}(x)}=\left(  \frac{1}{\varepsilon
_{1}\psi_{X}}\right)  \circ\psi_{X}^{-1}(x)\leqslant\frac{1}{l_{1}}.
\]
Note that if $\gamma_{1}>0$ we have $\varepsilon_{1}(x)\rightarrow\gamma_{1}$
so the second equality in the left-hand side yields back the sharper bound
$1/\gamma_{1}x$ used for (\ref{HXreg}).
\end{proof}

\begin{corollary}
Let $(C)$ hold with $\gamma>0$. Assume that $F$ and $G$ satisfy $(FG4)$
together with the condition of Proposition \ref{CSFG} with $\gamma_{1}>0$.
Assume that $(CFG)$ holds with $\theta>1$. Then the conclusion of Theorem
\ref{mainth} remains true.
\end{corollary}

\begin{proof}
The result was proved for $\gamma>0$ and $\theta>2$. The only changes needed
for $\theta>1$ are at cases 2 and 3 in the proof of Lemma \ref{lem:truc2} at
Section 5.1.2. We have $\psi_{Y}\geqslant\psi_{X}$, $\psi_{X}\in RV_{2}%
(\gamma_{1},m)$ and $\gamma_{1}\geqslant\gamma>0$ by (\ref{CFG1ab}). Applying
(\ref{HXreg}) from Proposition \ref{CSFG} yields%
\[
\sup_{u\in I_{n}}H(u)\leqslant\sup_{u\in I_{n}}\frac{2}{\gamma_{0}%
\log(1/(1-u))}\leqslant\frac{2}{(1-\beta)\gamma_{0}\log n}%
\]
and this extra $1/(\log n)$ makes the bounding sequence in (\ref{DL1}) tends
to $0$ provided that $1<\theta^{\prime}<\theta$ in (\ref{knbis}).
\end{proof}

\subsubsection{Consequences of (C2)}

\begin{proposition}
\label{C3rho}Assume $(C2)$. Then it holds%
\[
\rho(\left\vert x+\varepsilon\right\vert )-\rho(x)=k_{0}(x,\varepsilon
)\rho^{\prime}(x)\varepsilon
\]
where, for any $x_{0}>\tau_{1}$,%
\begin{equation}
\lim_{\delta_{0}\rightarrow0}\sup_{x>x_{0}}\sup_{\left\vert \varepsilon
\right\vert l^{\prime}(x)\leqslant\delta_{0}}\left\vert k_{0}(x,\varepsilon
)-1\right\vert =0. \label{DLexplicit}%
\end{equation}
In particular, there exists $\delta_{0}>0$ and $k_{0}>0$ such that, for all
$x>x_{0}$ and $\left\vert \varepsilon\right\vert \leqslant\delta_{0}%
/l^{\prime}(x)$ we have $\left\vert \rho(\left\vert x+\varepsilon\right\vert
)-\rho(x)\right\vert \leqslant k_{0}\rho^{\prime}(x)\left\vert \varepsilon
\right\vert .$
\end{proposition}

\begin{proof}
Fix $x_{0}>\tau_{1}>0$ and let $M>x_{0}$ be as large as needed below.\ If
$\varepsilon=0$ then (\ref{DLexplicit}) requires that $k_{0}(x,0)=1$ for
$x>x_{0}$. For $\varepsilon\neq0$ we distinguish between $x\in\left(
x_{0},M\right)  $ and $x\geqslant M$. In the first case, since $\rho
\in\mathcal{C}_{2}$ under $(C2)$ the Taylor expansion of $\rho$ holds
uniformely on $\left(  x_{0},M\right)  $. Namely, for any $\delta_{0}$ small
enough, $x\in\left(  x_{0},M\right)  $ and $\left\vert \varepsilon\right\vert
\leqslant\varepsilon_{0}=\delta_{0}/\inf\left\{  l^{\prime}(x):x\in\left(
x_{0},M\right)  \right\}  <x_{0}-\tau_{1}$ we have%
\[
\rho(\left\vert x+\varepsilon\right\vert )-\rho(x)=k_{0}(x,\varepsilon
)\rho^{\prime}(x)\varepsilon,\quad k_{0}(x,\varepsilon)=1+\frac{\rho
^{\prime\prime}(x^{\ast})}{2\rho^{\prime}(x)}\varepsilon,
\]
with $x^{\ast}\in\left(  x_{0}-\varepsilon_{0},M+\varepsilon_{0}\right)  $ and
$\left\vert k_{0}(x,\varepsilon)-1\right\vert \leqslant K\delta_{0}$ where
$K<+\infty$ depends on $x_{0},M,\rho$. We deduce that, for any $M>x_{0}$,%
\begin{equation}
\lim_{\delta_{0}\rightarrow0}\sup_{x_{0}<x<M}\sup_{\left\vert \varepsilon
\right\vert l^{\prime}(x)\leqslant\delta_{0}}\left\vert k_{0}(x,\varepsilon
)-1\right\vert =0. \label{xpetit}%
\end{equation}
If $x\geqslant M$ then $l^{\prime}(x)>0$ and we intend to expand%
\begin{equation}
\rho(\left\vert x+\varepsilon\right\vert )-\rho(x)=\rho(x)\left(
\exp(l(\left\vert x+\varepsilon\right\vert )-l(x))-1\right)  .
\label{rhoexpand2}%
\end{equation}
\textbf{Case 1.} Assume $(C2)$ with $\gamma>0$. Write $l(x)=x^{\gamma}L(x)$
where $L\in\mathcal{RV}_{2}(0,\tau_{1})$ satisfies (\ref{L'}). For any
$\delta_{0}\in\left(  0,\gamma l(M)/4\right)  $ define%
\begin{equation}
\Delta_{0}=\left\{  (x,\varepsilon):x\geqslant M,\ \left\vert \varepsilon
\right\vert l^{\prime}(x)\leqslant\delta_{0}\right\}  . \label{delta0}%
\end{equation}
By (\ref{phiprime}), for $M$ large enough and $(x,\varepsilon)\in\Delta_{0}$
it holds $l^{\prime}(x)>\gamma l(x)/2x$, which implies $\left\vert
\varepsilon\right\vert /x\leqslant2\delta_{0}/\gamma l(x)<1/2$ and $\left\vert
x+\varepsilon\right\vert =x+\varepsilon>M/2$. Therefore $\sup_{(x,\varepsilon
)\in\Delta_{0}}\left\vert \varepsilon\right\vert /x\rightarrow0$ as
$\delta_{0}\rightarrow0$ and we have, for $(x,\varepsilon)\in\Delta_{0}$,%
\begin{align}
\frac{l(x+\varepsilon)-l(x)}{x^{\gamma}}  &  =\left(  1+\frac{\varepsilon}%
{x}\right)  ^{\gamma}L(x+\varepsilon)-L(x)\nonumber\\
&  =\frac{\gamma\varepsilon}{x}(1+\delta_{1}(x,\varepsilon))L(x+\varepsilon
)+L(x+\varepsilon)-L(x) \label{phiprimeeps}%
\end{align}
where $\sup_{(x,\varepsilon)\in\Delta_{0}}\left\vert \delta_{1}(\varepsilon
,x)\right\vert \rightarrow0$ as $\delta_{0}\rightarrow0$. By (\ref{L'}) we
also have, for $(x,\varepsilon)\in\Delta_{0}$,%
\[
\left\vert L(x+\varepsilon)-L(x)\right\vert \leqslant\sup_{\left\vert
y-x\right\vert \leqslant\left\vert \varepsilon\right\vert }\left\vert
L^{\prime}(y)\right\vert \left\vert \varepsilon\right\vert =\sup_{\left\vert
y-x\right\vert \leqslant\left\vert \varepsilon\right\vert }\left\vert
\varepsilon_{1}(y)\right\vert \frac{L(y)}{y}\left\vert \varepsilon\right\vert
\]
where $\varepsilon_{1}(y)\rightarrow0$ as $y>x-\left\vert \varepsilon
\right\vert >M/2\rightarrow+\infty$. Moreover, for $\delta=2\delta_{0}/\gamma
l(M)$,%
\[
\frac{1}{L(x)}\sup_{\left\vert y-x\right\vert \leqslant\left\vert
\varepsilon\right\vert }L(y)=\sup_{1-\left\vert \varepsilon\right\vert
/x<\lambda<1+\left\vert \varepsilon\right\vert /x}\frac{L(\lambda x)}%
{L(x)}\leqslant\sup_{1-\delta<\lambda<1+\delta}\frac{L(\lambda x)}{L(x)}%
\]
and the second term has limit $1$ as $x\rightarrow+\infty$ since $L\in RV(0)$.
Hence for any $\varepsilon_{1}>0$, assuming $M$ so large that $\sup
_{y>M/2}\left\vert \varepsilon_{1}(y)\right\vert <\varepsilon_{1}/4$ and
$\delta_{0}$ small ensures that, for $(x,\varepsilon)\in\Delta_{0}$,%
\[
\left\vert L(x+\varepsilon)-L(x)\right\vert \leqslant\frac{\varepsilon_{1}}%
{3}\frac{L(x)}{x-\left\vert \varepsilon\right\vert }\left\vert \varepsilon
\right\vert \leqslant\frac{\varepsilon_{1}}{2}\frac{\left\vert \varepsilon
\right\vert }{x}L(x)
\]
and (\ref{phiprimeeps}) reads%
\[
l(x+\varepsilon)-l(x)=(1+\delta_{2}(x,\varepsilon))\frac{\gamma x^{\gamma
}L(x)}{x}\varepsilon=(1+\delta_{3}(x,\varepsilon))l^{\prime}(x)\varepsilon
\]
 then (\ref{rhoexpand2})
gives%
\[
\frac{\rho(x+\varepsilon)-\rho(x)}{l^{\prime}(x)\rho(x)\varepsilon}=\frac
{\exp(l(x+\varepsilon)-l(x))-1}{l^{\prime}(x)\varepsilon}=1+\delta
_{4}(x,\varepsilon)=k_{0}(x,\varepsilon)
\]
with $\sup_{(x,\varepsilon)\in\Delta_{0}}\left\vert \delta_{k}(\varepsilon
,x)\right\vert <\varepsilon_{1}$ for $k=2,3,4$.
We have proved that for any $\varepsilon_{1}>0$ there exists $M$ such that%
\[
\lim_{\delta_{0}\rightarrow0}\sup_{x\geqslant M}\sup_{\left\vert
\varepsilon\right\vert l^{\prime}(x)\leqslant\delta_{0}}\left\vert
k_{0}(x,\varepsilon)-1\right\vert \leqslant\varepsilon_{1}%
\]
which yields (\ref{DLexplicit}) when combined to (\ref{xpetit}).\smallskip

\noindent\textbf{Case 2.} Assume $(C2)$ with $\gamma=0$. Since $l\in
\mathcal{RV}_{2}^{+}(0,\tau_{1})$, (\ref{L'}) and (\ref{L1}) give%
\[
\varepsilon_{1}(x)=\frac{xl^{\prime}(x)}{l(x)}\geqslant\frac{l_{1}}{l(x)}>0
\]
where $\varepsilon_{1}(x)\rightarrow0$. Thus $l^{\prime}(x)\rightarrow0$ as
$x\rightarrow+\infty$ and $l^{\prime}\in\mathcal{RV}_{2}(-1,\tau_{1})$. Thus
$l^{\prime}$ is decreasing on $\left(  M,+\infty\right)  $ since $l^{\prime
}\in\mathcal{M}_{2}(\tau_{1},+\infty)$. Condider $\delta_{0}\in\left(
0,1/2l_{1}\right)  $ and define $\Delta_{0}$ as in (\ref{delta0}). For
$(x,\varepsilon)\in\Delta_{0}$ it holds%
\[
\frac{\left\vert \varepsilon\right\vert }{x}\leqslant l_{1}\frac{\left\vert
\varepsilon\right\vert }{x}\leqslant l(x)\varepsilon_{1}(x)\frac{\left\vert
\varepsilon\right\vert }{x}=\left\vert \varepsilon\right\vert l^{\prime
}(x)\leqslant\delta_{0}%
\]
hence $\left\vert x+\varepsilon\right\vert =x+\varepsilon>M/2$ again, and%
\[
l^{\prime}(x+\left\vert \varepsilon\right\vert )\left\vert \varepsilon
\right\vert \leqslant\left\vert l(x+\varepsilon)-l(x)\right\vert \leqslant
l^{\prime}(x-\left\vert \varepsilon\right\vert )\left\vert \varepsilon
\right\vert
\]
where, since $l^{\prime\prime}(x)\sim-l^{\prime}(x)/x$ by (\ref{phiseconde}),%
\begin{align*}
0  &  \leqslant\frac{l^{\prime}(x-\left\vert \varepsilon\right\vert
)-l^{\prime}(x)}{l^{\prime}(x-\left\vert \varepsilon\right\vert )}%
\leqslant\sup_{x-\left\vert \varepsilon\right\vert \leqslant y\leqslant
x}\frac{\left\vert l^{\prime\prime}(y)\varepsilon\right\vert }{l^{\prime}%
(y)}\leqslant\frac{\left\vert \varepsilon\right\vert }{x-\left\vert
\varepsilon\right\vert }\leqslant2\delta_{0},\\
0  &  \leqslant\frac{l^{\prime}(x)-l^{\prime}(x+\left\vert \varepsilon
\right\vert )}{l^{\prime}(x)}\leqslant\sup_{x\leqslant y\leqslant x+\left\vert
\varepsilon\right\vert }\frac{\left\vert l^{\prime\prime}(y)\varepsilon
\right\vert }{l^{\prime}(y)}\leqslant\frac{\left\vert \varepsilon\right\vert
}{x}\leqslant\delta_{0}.
\end{align*}
We deduce that for $k=1,2$ and $\sup_{(x,\varepsilon)\in\Delta_{0}}\left\vert
\delta_{k}(\varepsilon,x)\right\vert \rightarrow0$ as $\delta_{0}\rightarrow0$
it holds%
\[
l(x+\varepsilon)-l(x)=(1+\delta_{1}(x,\varepsilon))l^{\prime}(x)\varepsilon
\]
for all $(x,\varepsilon)\in\Delta_{0}$ and, by (\ref{rhoexpand2}),%
\[
\frac{\rho(x+\varepsilon)-\rho(x)}{l^{\prime}(x)\rho(x)\varepsilon}=\frac
{\exp(l(x+\varepsilon)-l(x))-1}{l^{\prime}(x)\varepsilon}=1+\delta
_{2}(x,\varepsilon)=k_{0}(x,\varepsilon)
\]
thus (\ref{DLexplicit}) follows.\smallskip
\end{proof}

\noindent Several arguments exploit the asymptotic convexity of $\rho$ which follows from$(C2)$:

\begin{proposition}
\label{convex}Under $(C2)\ $the function $\rho(x)$ is convex on $\left(
l_{2},+\infty\right)  $ for some $l_{2}>0$. If moreover $l_{1}>1$ it is
strictly convex.
\end{proposition}

\begin{proof}
We have to show that $\rho^{\prime\prime}(x)=(l^{\prime\prime}(x)+l^{\prime
}(x)^{2})\rho(x)\geqslant0$ if $(C2)$ holds. In the case $1\neq\gamma>0$ we
have, by (\ref{phiprime}) and (\ref{phiseconde}), as $x\rightarrow+\infty$,%
\[
l^{\prime}(x)\sim\frac{\gamma l(x)}{x},\quad l^{\prime\prime}(x)\sim
\frac{\gamma(\gamma-1)l(x)}{x^{2}}\ll l^{\prime}(x),\quad\frac{l^{\prime
\prime}(x)}{l^{\prime}(x)}\sim\frac{\gamma-1}{x},
\]
thus there exists $l_{2}>l^{-1}(1/\gamma)$ such that all $x>l_{2}$ satisfy
$l^{\prime}(x)>0$ and%
\[
l^{\prime\prime}(x)+l^{\prime}(x)^{2}\sim l^{\prime}(x)\left(  \frac{\gamma
-1}{x}+\frac{\gamma l(x)}{x}\right)  \geqslant\frac{l^{\prime}(x)}{x}\left(
\gamma l(x)-1\right)  >0.
\]
If $\gamma=1$ then $l(x)=xL(x)$ and $l^{\prime\prime}(x)=2L^{\prime
}(x)+xL^{\prime\prime}(x)\sim L^{\prime}(x)$ whereas $l^{\prime}(x)^{2}%
\sim(L(x)+xL^{\prime}(x))^{2}\sim L^{2}(x)$. Since $L^{\prime}(x)/L^{2}%
(x)=\varepsilon_{1}(x)/xL(x)\rightarrow0$ we have $l^{\prime\prime
}(x)+l^{\prime}(x)^{2}>0$ for $x>l_{2}$. If $\gamma=0$ in $(C2)$ then by
(\ref{L'}) and (\ref{L1}) we have%
\[
xl^{\prime}(x)=\varepsilon_{1}(x)l(x)\geqslant l_{1}\geqslant1,\quad
\varepsilon_{1}(x)\rightarrow0,
\]
hence%
\[
\frac{l_{1}}{x}\leqslant l^{\prime}(x)\leqslant\frac{l(x)}{x}%
\]
and $l^{\prime}\in RV_{1}^{+}(-1,0)$. Now by (\ref{phiseconde}) we get, as
$x\rightarrow+\infty$,%
\[
l^{\prime}(x)\sim\frac{\varepsilon_{1}(x)l(x)}{x},\quad l^{\prime\prime
}(x)\sim-\frac{l^{\prime}(x)}{x}=-\frac{\varepsilon_{1}(x)l(x)}{x^{2}}%
,\quad\frac{l^{\prime\prime}(x)}{l^{\prime}(x)}\sim-\frac{1}{x},
\]
so that%
\[
l^{\prime\prime}(x)+l^{\prime}(x)^{2}\sim\frac{l^{\prime}(x)}{x}\left(
\varepsilon_{1}(x)l(x)-1\right)  \geqslant\frac{l^{\prime}(x)}{x}\left(
l_{1}-1\right)  .
\]
Therefore if $l_{1}>1$ $\rho(x)$ is strictly convex on $\left(  l_{2}%
,+\infty\right)  $ for $l_{2}$ large enough. It remains convex for $l_{1}=1$.
\end{proof}

\bibliographystyle{plain}
\bibliography{TCLbib_HAL}
\end{document}